\documentclass[11pt,a4paper]{amsart}
\usepackage{graphicx}
\usepackage{amsmath,amsfonts}
\usepackage{amsthm,amssymb,latexsym}
\usepackage[active]{srcltx}
\usepackage[usenames]{color}

\newtheorem{theorem}{Theorem}[section]
\newtheorem{corollary}[theorem]{Corollary}
\newtheorem{lemma}[theorem]{Lemma}

\newtheorem*{fact*}{Fact}
\newtheorem{proposition}[theorem]{Proposition}
\theoremstyle{definition}

\newtheorem{remark}[theorem]{Remark}
\numberwithin{equation}{section}

\newcommand{\N}{\mathbb N}

\newcommand{\A}{\mathbb A}
\newcommand{\F}{\mathbb F}
\newcommand{\K}{\mathbb K}

\newcommand{\Pp}{\mathbb P}

\newcommand{\fq}{\F_{\hskip-0.7mm q}}

\newcommand{\cfq}{\overline{\F}_{\hskip-0.7mm q}}
\def\ifm#1#2{\relax \ifmmode#1\else#2\fi}

\begin{document}

\title[Value set of small families III]{On the
value set of small families of polynomials over a finite field, III}
\author[G. Matera et al.]{
Guillermo Matera${}^{1,2}$,
%
Mariana P\'erez${}^1$,
%
and Melina Privitelli${}^3$}

\address{${}^{1}$Instituto del Desarrollo Humano,
Universidad Nacional de Gene\-ral Sarmiento, J.M. Guti\'errez 1150
(B1613GSX) Los Polvorines, Buenos Aires, Argentina}
\email{\{gmatera,\,vperez\}@ungs.edu.ar}
\address{${}^{2}$ National Council of Science and Technology (CONICET),
Ar\-gentina}
\address{${}^{3}$Instituto de Ciencias,
Universidad Nacional de Gene\-ral Sarmiento, J.M. Guti\'errez 1150
(B1613GSX) Los Polvorines, Buenos Aires, Argentina}
\email{mprivite@ungs.edu.ar}
\thanks{The authors were partially supported by the grants
PIP CONICET 11220130100598, PIO CONICET-UNGS 14420140100027 and UNGS 30/3084.}%

\date{\today}
\begin{abstract}
We estimate the average cardinality $\mathcal{V}(\mathcal{A})$ of
the value set of a general family $\mathcal{A}$ of monic univariate
polynomials of degree $d$ with coefficients in the finite field
$\fq$. We establish conditions on the family $\mathcal{A}$ under which 
$\mathcal{V}(\mathcal{A})=\mu_d\,q+\mathcal{O}(q^{1/2})$, where
$\mu_d:=\sum_{r=1}^d{(-1)^{r-1}}/{r!}$. The result holds without any restriction
on the characteristic of $\fq$ and provides an explicit expression for the
constant underlying the $\mathcal{O}$--notation in terms of $d$. We
reduce the question to estimating the number of $\fq$--rational
points with pairwise--distinct coordinates of a certain family of
complete intersections defined over $\fq$. For this purpose, we
obtain an upper bound on the dimension of the singular locus of the
complete intersections under consideration, which allows us to
estimate the corresponding number of $\fq$--rational points. 
\end{abstract}
\maketitle

%
%
\section{Introduction}
Let $\fq$ be the finite field of $q:=p^s$ elements, where $p$ is a
prime number, let $\cfq$ denote its algebraic closure, and let $T$
be an indeterminate over $\cfq$. For $f\in\fq[T]$, its value set is
the image of the mapping from $\fq$ to $\fq$ defined by $f$ (cf.
\cite{LiNi83}). We shall denote its cardinality by $\mathcal{V}(f)$,
namely $\mathcal{V}(f):=|\{f(c):c\in\fq\}|$.

In a seminal paper, Birch and Swinnerton--Dyer \cite{BiSD59} showed
that, for fixed $d\ge 1$, if $f\in\fq[T]$ is a ``general''
polynomial of degree $d$, then
\begin{equation}\label{eq: intro: BirchSD}
\mathcal{V}(f)=\mu_d\,q+\mathcal{O}(q^{\frac{1}{2}}),
\end{equation}
where $\mu_d:=\sum_{r=1}^d{(-1)^{r-1}}/{r!}$ and the
$\mathcal{O}$--constant depends only on $d$.

Uchiyama \cite{Uchiyama55a} and Cohen (\cite{Cohen73},
\cite{Cohen72}) were concerned on estimates for the average
cardinality of the value set when $f$ ranges over all monic
polynomials of degree $d$ in $\fq[T]$. In particular, in
\cite{Cohen72} the problem of estimating the average cardinality of
the value set on linear families of monic polynomials of $\fq[T]$ of
degree $d$ is addressed. More precisely, it is shown that, for a
linear family $\mathcal{A}$ of codimension $m \le d-2$ satisfying
certain conditions,
\begin{equation}\label{eq: intro: Cohen72}
\mathcal{V}(\mathcal{A})=\mu_d\,q+ \mathcal{O}(q^{\frac{1}{2}}),
\end{equation}
where $\mathcal{V}(\mathcal{A})$ denotes the average cardinality of
the value set of the elements in $\mathcal{A}$. As a particular case
we have the classical case of polynomials with prescribed
coefficients, where simpler conditions are obtained.

A difficulty with \eqref{eq: intro: Cohen72} is that the hypotheses
on the linear family $\mathcal{A}$ seem complicated and not easy to
verify. A second concern is that \eqref{eq: intro: Cohen72} imposes
the restriction $p>d$, which inhibits its application to fields of
small characteristic. For these reasons, in \cite{CeMaPePr14} and
\cite{MaPePr14} we obtained explicit estimates for any family of
monic polynomials of $\fq[T]$ of degree $d$ with certain consecutive
coefficients prescribed, which are valid for $p>2$. In this paper we
develop a framework which allows us to significantly generalize
these results to rather general (eventually nonlinear) families of
monic polynomials of $\fq[T]$ of degree $d$.

More precisely, let $d$, $m$ be positive integers with $q>d\ge m+2$,
and let $A_{d-1},\ldots,A_0$ be indeterminates over $\cfq$. Let
$G_1,\ldots,G_m\in \fq [A_{d-1},\ldots,A_0]$ be polynomials of
degree $d_1,\ldots,d_m$ and
$\mathcal{A}:=\mathcal{A}(G_1,\ldots,G_m)$ the family
\begin{equation}\label{eq: definition A}
\mathcal{A}:=\Bigg\{T^d+\sum_{j=0}^{d-1}a_jT^j\in \fq[T]:
G_i(a_{d-1},\ldots,a_0)=0\ (1\le i\le m)\Bigg\}.
\end{equation}
Denote by $\mathcal{V}(\mathcal{A})$ the average value of
$\mathcal{V}(f)$ for $f$ ranging in $\mathcal{A}$, that is,
\begin{equation}\label{eq: value set in family}
\mathcal{V}(\mathcal{A}):= \frac{1}{|\mathcal{A}|}\sum_{f\in
\mathcal{A}}\mathcal{V}(f).
\end{equation}
Our main result establishes rather general conditions on
$G_1,\ldots, G_m$ under which the asymptotic behavior of
$\mathcal{V}(\mathcal{A})$ agrees with that of the general case, as
predicted in \eqref{eq: intro: BirchSD} and \eqref{eq: intro:
Cohen72}. More precisely, we prove that
$$
\left|\mathcal{V}(\mathcal{A})-\mu_d\,q\right|\le 2^d\delta(3D+d^2)
q^{\frac{1}{2}} + 67 \delta^2(D+2)^2\,{d^{d+5} e^{2 \sqrt{d}-d}},$$
where $\delta:=\prod_{i=1}^md_i$ and $D:=\sum_{i=1}^m(d_i-1)$.

Our approach relies on tools of algebraic geometry in the same vein
as \cite{CeMaPePr14} and \cite{MaPePr14}. In Section \ref{sec:
notions of algebraic geometry} we recall the basic notions and
notations of algebraic geometry we use. In Section \ref{sec:
combinatorial preliminaries} we provide a combinatorial expression
for $\mathcal{V}(\mathcal{A})$ in terms of the number
$\mathcal{S}_r^{\mathcal{A}}$ of certain ``interpolating sets'' with
$1\leq r\leq d$ and we relate each $\mathcal{S}_r^{\mathcal{A}}$
with the number of $\fq$--rational points of certain incidence
variety $\Gamma_r^*$ of $\cfq{\!}^{d+r}$. In Section \ref{section:
geometry of Gamma_r} we show that $\Gamma_r^*$ is an
$\fq$--definable normal complete intersection, and establish a
number of geometric properties of $\Gamma_r^*$. To estimate the
number of $\fq$--rational points of $\Gamma_r^*$ is necessary to
discuss the behavior of $\Gamma_r^*$ at ``infinity'', which is done
in Section \ref{section: geometry of pcl Gamma r}. Finally, the
results of Sections \ref{section: geometry of Gamma_r} and
\ref{section: geometry of pcl Gamma r} allow us to estimate, in
Section \ref{sec: number of points Gamma_j}, the number of
$\fq$--rational points of $\Gamma_r^*$, and therefore determine the
asymptotic behavior of $\mathcal{V}(\mathcal{A})$. Applications to
linear and nonlinear families of polynomials are briefly discussed.
%
%
\section{Basic notions of algebraic geometry}
\label{sec: notions of algebraic geometry}
In this section we collect the basic definitions and facts of
algebraic geometry that we need in the sequel. We use standard
notions and notations which can be found in, e.g., \cite{Kunz85},
\cite{Shafarevich94}.

Let $\K$ be any of the fields $\fq$ or $\cfq$. We denote by $\A^n$
the affine $n$--dimensional space $\cfq{\!}^{n}$ and by $\Pp^n$ the
projective $n$--dimensional space over $\cfq{\!}^{n+1}$. Both spaces
are endowed with their respective Zariski topologies over $\K$, for
which a closed set is the zero locus of a set of polynomials of
$\K[X_1,\ldots, X_{n}]$, or of a set of homogeneous polynomials of
$\K[X_0,\ldots, X_{n}]$.

A subset $V\subset \Pp^n$ is a {\em projective variety defined over}
$\K$ (or a projective $\K$--variety for short) if it is the set of
common zeros in $\Pp^n$ of homogeneous polynomials $F_1,\ldots, F_m
\in\K[X_0 ,\ldots, X_n]$. Correspondingly, an {\em affine variety of
$\A^n$ defined over} $\K$ (or an affine $\K$--variety) is the set of
common zeros in $\A^n$ of polynomials $F_1,\ldots, F_{m} \in
\K[X_1,\ldots, X_{n}]$. We think a projective or affine
$\K$--variety to be equipped with the induced Zariski topology. We
shall denote by $\{F_1=0,\ldots, F_m=0\}$ or $V(F_1,\ldots,F_m)$ the
affine or projective $\K$--variety consisting of the common zeros of
$F_1,\ldots, F_m$.

In the remaining part of this section, unless otherwise stated, all
results referring to varieties in general should be understood as
valid for both projective and affine varieties.

A $\K$--variety $V$ is {\em irreducible} if it cannot be expressed
as a finite union of proper $\K$--subvarieties of $V$. Further, $V$
is {\em absolutely irreducible} if it is $\cfq$--irreducible as a
$\cfq$--variety. Any $\K$--variety $V$ can be expressed as an
irredundant union $V=\mathcal{C}_1\cup \cdots\cup\mathcal{C}_s$ of
irreducible (absolutely irreducible) $\K$--varieties, unique up to
reordering, which are called the {\em irreducible} ({\em absolutely
irreducible}) $\K$--{\em components} of $V$.

For a $\K$--variety $V$ contained in $\Pp^n$ or $\A^n$, we denote by
$I(V)$ its {\em defining ideal}, namely the set of polynomials of
$\K[X_0,\ldots, X_n]$, or of $\K[X_1,\ldots, X_n]$, vanishing on
$V$. The {\em coordinate ring} $\K[V]$ of $V$ is defined as the
quotient ring $\K[X_0,\ldots,X_n]/I(V)$ or
$\K[X_1,\ldots,X_n]/I(V)$. The {\em dimension} $\dim V$ of a
$\K$-variety $V$ is the length $r$ of a longest chain
$V_0\varsubsetneq V_1 \varsubsetneq\cdots \varsubsetneq V_r$ of
nonempty irreducible $\K$--varieties contained in $V$. A
$\K$--variety $V$ is called {\em equidimensional} if all the
irreducible $\K$--components of $V$ are of the same dimension. In
such a case, we say that $V$ has {\em pure dimension} $r$, meaning
that every irreducible $\K$--component of $V$ has dimension $r$.

A $\K$--variety of $\Pp^n$ or $\A^n$ of pure dimension $n-1$ is
called a $\K$--hypersurface. It turns out that a $\K$--hypersurface
of $\Pp^n$ (or $\A^n$) is the set of zeros of a single nonzero
polynomial of $\K[X_0,\ldots, X_n]$ (or of $\K[X_1,\ldots, X_n]$).

The {\em degree} $\deg V$ of an irreducible $\K$--variety $V$ is the
maximum number of points lying in the intersection of $V$ with a
linear space $L$ of codimension $\dim V$, for which $V\cap L$ is a
finite set. More generally, following \cite{Heintz83} (see also
\cite{Fulton84}), if $V=\mathcal{C}_1\cup\cdots\cup \mathcal{C}_s$
is the decomposition of $V$ into irreducible $\K$--components, we
define the degree of $V$ as
$$\deg V:=\sum_{i=1}^s\deg \mathcal{C}_i.$$
The degree of a $\K$--hypersurface $V$ is the degree of a polynomial
of minimal degree defining $V$. Another property is that the degree
of a dense open subset of a $\K$--variety $V$ is equal to the degree
of $V$.

An important tool for our estimates is the following {\em B\'ezout
inequality} (see \cite{Heintz83}, \cite{Fulton84}, \cite{Vogel84}):
if $V$ and $W$ are $\K$--varieties of the same ambient space, then
the following inequality holds:
\begin{equation}\label{eq: Bezout}
\deg (V\cap W)\le \deg V \cdot \deg W.
\end{equation}

Let $V\subset\A^n$ be a $\K$--variety and $I(V)\subset
\K[X_1,\ldots, X_n]$ its defining ideal. Let $x$ be a point of $V$.
The {\em dimension} $\dim_xV$ {\em of} $V$ {\em at} $x$ is the
maximum of the dimensions of the irreducible $\K$--components of $V$
that contain $x$. If $I(V)=(F_1,\ldots, F_m)$, the {\em tangent
space} $\mathcal{T}_xV$ to $V$ at $x$ is the kernel of the Jacobian
matrix $(\partial F_i/\partial X_j)_{1\le i\le m,1\le j\le n}(x)$ of
$F_1,\ldots, F_m$ with respect to $X_1,\ldots, X_n$ at $x$. We have
the inequality $\dim\mathcal{T}_xV\ge \dim_xV$ (see, e.g.,
\cite[page 94]{Shafarevich94}). The point $x$ is {\em regular} if
$\dim\mathcal{T}_xV=\dim_xV$. Otherwise, the point $x$ is called
{\em singular}. The set of singular points of $V$ is the {\em
singular locus} $\mathrm{Sing}(V)$ of $V$; it is a closed
$\K$--subvariety of $V$. A variety is called {\em nonsingular} if
its singular locus is empty. For a projective variety, the concepts
of tangent space, regular and singular point can be defined by
considering an affine neighborhood of the point under consideration.

Let $V$ and $W$ be irreducible affine $\K$--varieties of the same
dimension and let $f:V\to W$ be a regular map for which
$\overline{f(V)}=W$ holds, where $\overline{f(V)}$ denotes the
closure of $f(V)$ with respect to the Zariski topology of $W$. Such
a map is called {\em dominant}. Then $f$ induces a ring extension
$\K[W]\hookrightarrow \K[V]$ by composition with $f$. We say that
the dominant map $f$ is {\em finite} if this extension is integral,
namely each element $\eta\in\K[V]$ satisfies a monic equation with
coefficients in $\K[W]$. A basic fact is that a dominant finite
morphism is necessarily closed. Another fact concerning dominant
finite morphisms we shall use is that the preimage $f^{-1}(S)$ of an
irreducible closed subset $S\subset W$ is of pure dimension $\dim S$
(see, e.g., \cite[\S 4.2, Proposition]{Danilov94}).
%
%
\subsection{Rational points}
Let $\Pp^n(\fq)$ be the $n$--dimensional projective space over $\fq$
and let $\A^n(\fq)$ be the $n$--dimensional $\fq$--vector space
$\fq^n$. For a projective variety $V\subset\Pp^n$ or an affine
variety $V\subset\A^n$, we denote by $V(\fq)$ the set of
$\fq$--rational points of $V$, namely $V(\fq):=V\cap \Pp^n(\fq)$ in
the projective case and $V(\fq):=V\cap \A^n(\fq)$ in the affine
case. For an affine variety $V$ of dimension $r$ and degree
$\delta$, we have (see, e.g., \cite[Lemma 2.1]{CaMa06})
\begin{equation}\label{eq: upper bound -- affine gral}
   |V(\fq)|\leq \delta q^r.
\end{equation}
On the other hand, if $V$ is a projective variety of dimension $r$
and degree $\delta$, we have (see \cite[Proposition 12.1]{GhLa02a}
or \cite[Proposition 3.1]{CaMa07}; see \cite{LaRo15} for more
precise upper bounds)
 \begin{equation}\label{eq: upper bound -- projective gral}
   |V(\fq)|\leq \delta\, p_r,
 \end{equation}
where $p_r:=q^r+q^{r-1}+\cdots+q+1=|\Pp^r(\fq)|$.
%
%
\subsection{Complete intersections}
Elements $F_1,\ldots, F_{n-r}$ in $\mathbb{K}[X_1,\ldots,X_n]$ or
$\mathbb{K}[X_0,\ldots,X_n]$ form a \emph{regular sequence} if $F_1$
is nonzero and no $F_i$ is zero or a zero divisor in the quotient
ring $\mathbb{K}[X_1,\ldots,X_n]/ (F_1,\ldots,F_{i-1})$ or
$\mathbb{K}[X_0,\ldots,X_n]/ (F_1,\ldots,F_{i-1})$ for $2\leq i \leq
n-r$. In that case, the (affine or projective) $\mathbb{K}$--variety
$V:=V(F_1,\ldots,F_{n-r})$ is called a {\em set--theoretic complete
intersection}. We remark that $V$ is necessarily of pure dimension
$r$. Furthermore, $V$ is called an {\em ideal--theoretic complete
intersection} if its ideal $I(V)$ over $\K$ can be generated by
$n-r$ polynomials.

If $V\subset\Pp^n$ is an ideal--theoretic complete intersection
defined over $\K$ of dimension $r$, and $F_1 ,\ldots, F_{n-r}$ is a
system of homogeneous generators of $I(V)$, the degrees $d_1,\ldots,
d_{n-r}$ depend only on $V$ and not on the system of generators.
Arranging the $d_i$ in such a way that $d_1\geq d_2 \geq \cdots \geq
d_{n-r}$, we call $(d_1,\ldots, d_{n-r})$ the {\em multidegree} of
$V$. In this case, a stronger version of 
\eqref{eq: Bezout} holds, called the {\em B\'ezout theorem} (see,
e.g., \cite[Theorem 18.3]{Harris92}):
$$\deg V=d_1\cdots d_{n-r}.$$

In what follows we shall deal with a particular class of complete
intersections, which we now define. A complete intersection $V$ is
called {\em normal} if it is {\em regular in codimension 1}, that
is, the singular locus $\mathrm{Sing}(V)$ of $V$ has codimension at
least $2$ in $V$, namely $\dim V-\dim \mathrm{Sing}(V)\ge 2$
(actually, normality is a general notion that agrees on complete
intersections with the one we define here). A fundamental result for
projective complete intersections is the Hartshorne connectedness
theorem (see, e.g., \cite[Theorem VI.4.2]{Kunz85}): if
$V\subset\Pp^n$ is a complete intersection defined over $\K$ and
$W\subset V$ is any $\K$--subvariety of codimension at least 2, then
$V\setminus W$ is connected in the Zariski topology of $\Pp^n$ over
$\K$. Applying the Hartshorne connectedness theorem with
$W:=\mathrm{Sing}(V)$, one deduces the following result.
\begin{theorem}\label{theorem: normal complete int implies irred}
If $V\subset\Pp^n$ is a normal complete intersection, then $V$ is
absolutely irreducible.
\end{theorem}
%
%
\section{A geometric approach to estimate value sets}
\label{sec: combinatorial preliminaries}
Let $m$ and $d$ be positive integers with $q>d\ge m+2$, and let
$\mathcal{A}$ be the family of \eqref{eq: definition A}. We may
assume without loss of generality that $G_1,\ldots, G_m$ are
elements of $\fq[A_{d-1},\ldots, A_1]$. Indeed, let
$\Pi:\mathcal{A}\to\fq$ be the mapping
$\Pi(T^d+a_{d-1}T^{d-1}+\cdots+a_0):=a_0$. Denote
$\mathcal{A}_{a_0}:=\Pi^{-1}(a_0)$. We have
$$\frac{1}{|\mathcal{A}|}\sum_{f\in
\mathcal{A}}\mathcal{V}(f)-\mu_dq=
\frac{1}{\sum\limits_{a_0\in\fq}|\mathcal{A}_{a_0}|}\sum\limits_{a_0\in\fq}
|\mathcal{A}_{a_0}|\Bigg(\frac{1}{|\mathcal{A}_{a_0}|}\sum_{f\in
\mathcal{A}_{a_0}}\mathcal{V}(f)-\mu_dq\Bigg).$$
As a consequence, if there exists a constant $E(d_1,\ldots, d_m,d)$
such that
$$\Bigg|\frac{1}{|\mathcal{A}_{a_0}|}\sum_{f\in
\mathcal{A}_{a_0}}\mathcal{V}(f)-\mu_dq\Bigg|\le E(d_1,\ldots,
d_m,d)q^{\frac{1}{2}}
$$
holds for any $a_0\in\fq$, then we conclude that
\begin{align*}
\Bigg|\frac{1}{|\mathcal{A}|}\sum_{f\in
\mathcal{A}}\mathcal{V}(f)-\mu_dq\Bigg|&\le
\frac{1}{\sum\limits_{a_0\in\fq}|\mathcal{A}_{a_0}|}
\sum\limits_{a_0\in\fq} |\mathcal{A}_{a_0}|E(d_1,\ldots,
d_m,d)q^{\frac{1}{2}}\\&\le E(d_1,\ldots, d_m,d)q^{\frac{1}{2}}.
\end{align*}
Further, as $\mathcal{V}(f)=\mathcal{V}(f+a_0)$ for any
$f\in\mathcal{A}$, we shall also assume that $f(0)=0$ for any
$f\in\mathcal{A}$.

Observe that, given $f\in \mathcal{A}$, $\mathcal{V}(f)$ equals the
number of $a_0\in \fq$ for which $f+a_0$ has at least one root in
$\fq$. If $\K$ is any of the fields $\fq$ or $\cfq$, by $\K[T]_d$ we
denote the set of monic polynomials of $\K[T]$ of degree $d$. Let
$\mathcal{N}: \fq[T]_d \to \mathbb{Z}_{\geq 0 }$ be the counting
function of the number of roots in $\fq$ and
$\boldsymbol{1}_{\{\mathcal{N}>0\}}:\fq[T]_d \to \{0,1\}$ the
characteristic function of the set of polynomials having at least
one root in $\fq$. We deduce that
\begin{align*}
\sum_{f\in \mathcal{A}}\mathcal{V}(f)&= \sum_{a_0\in\fq}\sum_{f\in
\mathcal{A}} \boldsymbol{1}_{\{\mathcal{N}>0\}}
(f+a_0)\\&=\big|\{f+a_0\in\mathcal{A}+ \fq: \mathcal{N}(f+a_0)>0
\}\big|.
\end{align*}

For a set $\mathcal{X}\subset\fq$, we define
$\mathcal{S}_{\mathcal{X}}^{\mathcal{A}} \subset \fq[T]_d$ as the
set of polynomials  $f+a_0 \in \mathcal{A}+\fq$ vanishing on
$\mathcal{X}$, namely
$$
\mathcal{S}_{\mathcal{X}}^{\mathcal{A}}:=\{f+a_0 \in
\mathcal{A}+\fq:\, (f+a_0)(x)=0\textrm{ for any }x\in\mathcal{X}\}.
$$
For $r\in\N$ we shall use the symbol $\mathcal{X}_r$ to denote a
subset of $\fq$ of $r$ elements. Our approach to determine the
asymptotic behavior of $\mathcal{V}(\mathcal{A})$ relies on the
following combinatorial result.
\begin{lemma}\label{lemma: reduction to interp sets}
Given $d,m\in\N$ with $q>d\ge m+2$, we have
\begin{equation}\label{eq: our formula for value sets}
\mathcal{V}(\mathcal{A})=\frac{1}{|\mathcal{A}|}
\sum_{r=1}^d(-1)^{r-1}\sum_{\mathcal{X}_r \subset \fq}|
\mathcal{S}_{\mathcal{X}_r}^{\mathcal{A}}|.
\end{equation}
\end{lemma}
\begin{proof}
Given a subset $\mathcal{X}_r:=\{\alpha_1,\ldots,
\alpha_r\}\subset\fq$, consider the set
$\mathcal{S}_{\mathcal{X}_r}^{\mathcal{A}} \subset\fq[T]_d$ defined
as above. It is easy to see that
$\mathcal{S}_{\mathcal{X}_r}^{\mathcal{A}}=
\bigcap_{i=1}^{r}\mathcal{S}_{\{\alpha_i\}}^{\mathcal{A}}$ and
$$
\big|\{f+a_0\in \mathcal{A}+\fq:\,\, \mathcal{N}(f+a_0)>0
\}\big|=\Bigg\vert \bigcup_{x\in\fq}
\mathcal{S}_{\{x\}}^{\mathcal{A}}\Bigg\vert.
$$
Therefore the inclusion--exclusion principle implies
$$
\mathcal{V}(\mathcal{A})=\frac{1}{|\mathcal{A}|}\Bigg|\bigcup_{x\in
\fq}\mathcal{S}_{\{x\}}^{\mathcal{A}}\Bigg|=\frac{1}{|\mathcal{A}|}
\sum_{r=1}^{q}(-1)^{r-1}\sum_{\mathcal{X}_r \subset \fq}|
\mathcal{S}_{\mathcal{X}_r}^{\mathcal{A}}|.
$$
Now $|\mathcal{S}_{\mathcal{X}_r}^{\mathcal{A}}|=0$ for $r>d$,
because a polynomial of degree $d$ cannot vanish on more than $d$
elements of $\fq$. This readily implies the lemma.
\end{proof}

Lemma \ref{lemma: reduction to interp sets} shows that the behavior
of $\mathcal{V}(\mathcal{A})$ is determined by that of
\begin{equation}\label{eq:number S_j(A)}
\mathcal{S}_r^{\mathcal{A}}:=\sum_{\mathcal{X}_r \subset \fq}|
\mathcal{S}_{\mathcal{X}_r}^{\mathcal{A}}|,
\end{equation}
for $1\leq r \leq d$, which are the subject of the next sections.
%
%
\subsection{A geometric approach to estimate
$\mathcal{S}_r^{\mathcal{A}}$} \label{sec: geometric approach}
Fix $r$ with $1\leq r\leq d$. Let $A_{d-1},\ldots, A_0$ be
indeterminates over $\cfq$ and let
$G_1,\ldots,G_m\in\fq[A_{d-1},\ldots,A_1]$ be the polynomials
defining the family $\mathcal{A}$ of \eqref{eq: definition A}. Set
$\boldsymbol A:=(A_{d-1},\dots,A_1)$ and $\boldsymbol
A_0:=(\boldsymbol A,A_0)$. To estimate
$\mathcal{S}_r^{\mathcal{A}}$, we introduce the following
definitions and notations. Let $T,T_1,\ldots, T_r$ be new
indeterminates over $\cfq$ and denote $\boldsymbol
T:=(T_1,\ldots,T_r)$. Consider the polynomial $F \in \fq[\boldsymbol
A_0,T]$ defined as
\begin{equation}\label{eq: polynomial F}
F(\boldsymbol A_0, T):=T^d+A_{d-1}T^{d-1}+\dots+A_1T+A_0.
\end{equation}
Observe that if $\boldsymbol a_0\in \fq^d$, then we may write
$F(\boldsymbol a_0,T)=f+a_0$, where $f\in\fq[T]_d$ and $f(0)=0$.

Consider the affine quasi--$\fq$--variety $\Gamma_r \subset
\A^{d+r}$ defined as follows:
\begin{align*}
\Gamma_{r}:=\{(\boldsymbol a,a_0,
\boldsymbol\alpha)\in\A^{d}\times\A^r:\ \alpha_i &\neq \alpha_j\ (1
\leq i<j \leq r),\\ F(\boldsymbol a_0,\alpha_i)&=0\ (1 \leq i \leq
r), \ G_k(\boldsymbol a)=0 \, (1 \leq k \leq m)\}.
\end{align*}
Our next result explains how the number $|\Gamma_r(\fq)|$ of
$\fq$--rational points of $\Gamma_r$ is related to the numbers
$\mathcal{S}_r^\mathcal{A}$ $(1\le r\le d)$.
\begin{lemma} \label{lemma: relacion entre gamma m,n y S m,n}
Let  $r$ be an integer with $1\leq r\leq d$. Then
$$\frac{|\Gamma_{r}(\fq)|}{r!}=\mathcal{S}_r^{\mathcal{A}}.$$
\end{lemma}
\begin{proof}
Let $(\boldsymbol a_0, \boldsymbol\alpha)$ be a point of
$\Gamma_{r}(\fq)$ and $\sigma:\{1,\dots,r\}\to\{1,\dots,r\}$ an
arbitrary permutation. Let $\sigma(\boldsymbol\alpha)$ be the image
of $\boldsymbol{\alpha}$ by the linear mapping induced by $\sigma$.
Then it is clear that $\big(\boldsymbol a_0,
\sigma(\boldsymbol\alpha)\big)$ belong to $\Gamma_{r}(\fq)$.
Furthermore, $\sigma(\boldsymbol\alpha)=\boldsymbol{\alpha}$ if and
only if $\sigma$ is the identity permutation. This shows that
$\mathbb{S}_r$, the symmetric group
 of $r$ elements, acts over
the set $\Gamma_{r}(\fq)$ and each orbit under this action has $r!$
elements.

The orbit of an arbitrary point $(\boldsymbol a_0,
\boldsymbol{\alpha}) \in \Gamma_r(\fq)$ uniquely determines a
polynomial $F(\boldsymbol a_0,T)=f+a_0$ with $f \in \mathcal{A}$ and
a set $\mathcal{X}_r:=\{\alpha_1,\dots,\alpha_r\}\subset\fq$ with
$|\mathcal{X}_r|=r$ such that $(f+a_0)|_{\mathcal{X}_r}\equiv 0$.
Therefore, each orbit uniquely determines a set
$\mathcal{X}_r\subset\fq$ with $|\mathcal{X}_r|=r$  and an element
of $\mathcal{S}_{\mathcal{X}_r}^{\mathcal{A}}$. Reciprocally, to
each element of $\mathcal{S}_{\mathcal{X}_r}^{\mathcal{A}}$
corresponds a unique orbit of $\Gamma_{r}(\fq)$. This implies
$$\mbox{number of
orbits of }\Gamma_{r}(\fq)=\sum_{\mathcal{X}_r \subseteq \fq}|
\mathcal{S}_{\mathcal{X}_r}^{\mathcal{A}}|,$$
and finishes the proof of the lemma.
\end{proof}

To estimate the quantity $|\Gamma_{r}(\fq)|$ we shall consider the
Zariski closure $\mathrm{cl}(\Gamma_r)$ of $\Gamma_r\subset
\A^{d+r}$. Our aim is to provide explicit equations defining
$\mathrm{cl}(\Gamma_r)$. For this purpose, we shall use the
following notation. Let $X_1,\ldots, X_{l+1}$ be indeterminates over
$\cfq$ and $f\in\cfq[T]$ a polynomial of degree at most $l$. For
notational convenience, we define the 0th divided difference
$\Delta^0f\in\cfq[X_1]$ of $f$ as $\Delta^0f:=f(X_1)$. Further, for
$1\le i\le l$ we define the $i$th divided difference
$\Delta^if\in\cfq[X_1,\ldots, X_{i+1}]$ of $f$ as
$$\Delta^if(X_1,\ldots,X_{i+1})=\dfrac{\Delta^{i-1}f(X_1,\ldots,X_i)-
\Delta^{i-1}f(X_1,\ldots,X_{i-1},X_{i+1})}{X_i-X_{i+1}}.$$

With these notations, let $\Gamma^*_{r}\subset \A^{d+r}$ be the
$\fq$--variety defined as
\begin{align*}
\Gamma_{r}^*:=\{(\boldsymbol a_0,\boldsymbol\alpha)
\in\A^{d}\times\A^r:\ \Delta^{i-1}F(\boldsymbol a_0,
\alpha_1,\ldots,\alpha_i)=0\ &(1\leq i\leq r),\\ G_k(\boldsymbol
a_0)=0 \ &(1 \leq k \leq m) \},
\end{align*}
where $\Delta^{i-1}F(\boldsymbol a_0,T_1,\ldots,T_i)$  denotes the
$(i-1)$--divided difference of $F(\boldsymbol a_0,T)\in \cfq[T]$.
Next we relate the varieties $\Gamma_{r}$ and $\Gamma_{r}^*$ .
\begin{lemma}\label{lemma: relacion gamma_j y gamma_j estrella}
With notations and assumptions as above, we have
\begin{equation}\label{eq: relacion gamma_j y gamma_j estrella}
\Gamma_{r}=\Gamma_{r}^*\cap\{(\boldsymbol{a}_0,\boldsymbol{\alpha}):\alpha_i\neq\alpha_j\
(1\le i<j\le r)\}.
\end{equation}
\end{lemma}
\begin{proof}
Let $(\boldsymbol a_0, \boldsymbol\alpha)$ be a point  of
$\Gamma_r$. By the definition of the divided differences of
$F(\boldsymbol a_0,T)$ we easily conclude that $(\boldsymbol a_0,
\boldsymbol\alpha)\in \Gamma_r^*$. On the other hand, let
$(\boldsymbol a_0, \boldsymbol\alpha)$ be a point belonging to the
set in the right--hand side of (\ref{eq: relacion gamma_j y gamma_j
estrella}). We claim that $F(\boldsymbol a_0,\alpha_i)=0$ for $1\le
i\le r$. We observe that $F(\boldsymbol a_0,\alpha_1)=
\Delta^0F(\boldsymbol a_0,\alpha_1)=0$. Arguing inductively, suppose
that we have $F(\boldsymbol a_0,\alpha_1)=\cdots =F(\boldsymbol
a_0,\alpha_{i-1})=0$. By definition $\Delta^{i-1}F(\boldsymbol a_0,
\alpha_1,\ldots,\alpha_i)$ can be expressed as a linear combination
with nonzero coefficients of the differences $F(\boldsymbol
a_0,\alpha_{j+1})-F(\boldsymbol a_0,\alpha_j)$ with $1\le j\le i-1$.
Combining the inductive hypothesis with the fact that
$\Delta^{i-1}F(\boldsymbol a_0, \alpha_1,\ldots,\alpha_i)=0$, we
easily conclude that $F(\boldsymbol a_0,\alpha_i)=0$, finishing thus
the proof of the claim.
\end{proof}
%
%
\section{On the geometry of the variety $\Gamma_{r}^*$}
\label{section: geometry of Gamma_r}
In this section we establish several properties of the geometry of
the affine $\fq$--variety $\Gamma_{r}^{*}$, assuming that the
polynomials $G_1,\ldots, G_m$ and the affine variety $V\subset\A^d$
they define satisfy certain the conditions that we now state. The
first conditions allow us to estimate the cardinality of
$\mathcal{A}$:
\begin{itemize}
\item[(${\sf H_1}$)] $G_1,\ldots, G_m$ form a regular
sequence and generate a radical ideal of $\fq[A_{d-1},\ldots, A_0]$.
\item[(${\sf H_2}$)] The variety $V\subset\A^d$
defined by $G_1,\ldots, G_m$ is normal.
\item[(${\sf H_3}$)] Let $G_1^{d_1},\ldots, G_m^{d_m}$ denote the
homogeneous parts of higher degree of $G_1,\ldots, G_m$. Then
$G_1^{d_1},\ldots, G_m^{d_m}$ satisfy (${\sf H_1}$) and (${\sf
H_2}$).
\end{itemize}

As stated in the introduction, we are interested in families
$\mathcal{A}$ for which $\mathcal{V}(\mathcal{A})
=\mu_dq+\mathcal{O}(q^{\frac{1}{2}})$. If many of the polynomials in
$\mathcal{A}+\fq$ are not square--free, then
$\mathcal{V}(\mathcal{A})$ might not behave as in the general case.
For $\mathcal{B}\subset\cfq[T]_d$, the set of elements of
$\mathcal{B}$ which are not square--free is called the {\em
discriminant locus} $\mathcal{D}(\mathcal{B})$ of $\mathcal{B}$.
With a slight abuse of notation, in what follows we identify each
$f_{\boldsymbol a_0}=T^d+a_{d-1}T^{d-1}+\dots+a_0\in\mathcal{B}$
with the $d$--tuple $\boldsymbol a_0:=(a_{d-1},\ldots, a_0)$, and
consider $\mathcal{B}$ as a subset of $\A^d$. For $f_{\boldsymbol
a_0}\in \mathcal{B}$, let $\mathrm{Disc}(f_{\boldsymbol
a_0}):=\mathrm{Res}(f_{\boldsymbol a_0},f_{\boldsymbol a_0}')$
denote the discriminant of $f_{\boldsymbol a_0}$, that is, the
resultant of $f_{\boldsymbol a_0}$ and its derivative
$f_{\boldsymbol a_0}'$. Since $f_{\boldsymbol a_0}$ has degree $d$,
by basic properties of resultants it follows that
\begin{align*}
\mathrm{Disc}(f_{\boldsymbol a_0})&=\mathrm{Disc}(F(\boldsymbol
A_0,T))|_{\boldsymbol A_0=\boldsymbol
a_0}\\&:=\mathrm{Res}(F(\boldsymbol A_0,T), \Delta^1F(\boldsymbol
A_0,T,T),T)|_{\boldsymbol A_0=\boldsymbol a_0},
\end{align*}
where the expression $\mathrm{Res}$ in the right--hand side denotes
resultant with respect to $T$. Observe that
$\mathcal{D}(\mathcal{B})=\{\boldsymbol a_0\in
\mathcal{B}:\mathrm{Disc}(F(\boldsymbol A_0,T))|_{\boldsymbol
A_0=\boldsymbol a_0}=0\}$.

We shall need further to consider first subdiscriminant loci. The
{\em first subdiscriminant locus} $\mathcal{S}_1(\mathcal{B})$ of
$\mathcal{B}\subset\cfq[T]_d$ is the set of $\boldsymbol
a_0\in\mathcal{B}$ for which the first subdiscriminant
$\mathrm{Subdisc}(f_{\boldsymbol
a_0}):=\mathrm{Subres}(f_{\boldsymbol a_0}, f'_{\boldsymbol a_0})$
vanishes, where $\mathrm{Subres}(f_{\boldsymbol a_0},
f'_{\boldsymbol a_0})$ denotes the first subresultant of
$f_{\boldsymbol a_0}$ and $f'_{\boldsymbol a_0}$. Since
$f_{\boldsymbol a_0}$ has degree $d$, basic properties of
subresultants imply
\begin{align*}
\mathrm{Subdisc}(f_{\boldsymbol
a_0})&=\mathrm{Subdisc}(F(\boldsymbol A_0,T))|_{\boldsymbol
A_0=\boldsymbol a_0}\\&:=\mathrm{Subres}(F(\boldsymbol A_0,T),
\Delta^1F(\boldsymbol A_0,T,T),T))|_{\boldsymbol A_0=\boldsymbol
a_0},
\end{align*}
where $\mathrm{Subres}$ in the right--hand side denotes first
subresultant with respect to $T$. We have
$\mathcal{S}_1(\mathcal{B})=\{\boldsymbol a_0\in
\mathcal{B}:\mathrm{Subdisc}(F(\boldsymbol A_0,T))|_{\boldsymbol
A_0=\boldsymbol a_0}=0\}$.

Our next condition requires that the discriminant and the first
subdiscriminant locus intersect well $V$. More precisely, we require
the condition:
\begin{itemize}
\item[(${\sf H_4}$)]
$V\cap\mathcal{D}(V)$ has codimension one in $V$, and
$V\cap\mathcal{D}(V)\cap\mathcal{S}_1(V)$ has codimension two in
$V$.
\end{itemize}

We shall prove that $\Gamma_{r}^{*}$ is a set--theoretic complete
intersection, whose singular locus has codimension at least $2$.
This will allow us to conclude that $\Gamma_{r}^{*}$ is an
ideal--theoretic complete intersection.
\begin{lemma}\label{lemma: gamma j is set-theoretic complete intersection}
$\Gamma_{r}^{*}$ is a set--theoretic complete intersection of
dimension $d-m$.
\end{lemma}
\begin{proof}
By hypothesis $({\sf H_1})$, $G_1,\ldots, G_m$ form a regular
sequence. In order to prove that $G_1,\ldots, G_m,\Delta^{i-1}F(\boldsymbol
A_0,T_1,\ldots,T_i)$ $(1\le i\le r)$ form a regular sequence, we
argue by induction on $i$.

For $i=1$, we observe that the set of common zeros of $G_1,\ldots,
G_m$ in $\A^d\times \A^r$ is $V\times\A^r$, and each irreducible
component of $V\times\A^r$ is of the form $\mathcal{C}\times\A^r$,
where $\mathcal{C}$ is an irreducible component of $V$. As
$\Delta^0F(\boldsymbol A_0,T_1)=F(\boldsymbol A_0,T_1)$ is of degree
$d$ in $T_1$, it cannot vanish identically on any component
$\mathcal{C}\times\A^r$, which implies that it cannot be a zero
divisor modulo $G_1,\ldots, G_m$.

Now suppose that the assertion is proved for $1\le j\le r-1$, that
is, the polynomials $G_1,\ldots, G_m,\Delta^{i-1}F(\boldsymbol A_0,T_1,\ldots,T_i)$
$(1\le i\le j)$ form a regular sequence. These are all elements of
$\fq[\boldsymbol A_0,T_1,\ldots, T_j]$. On the other hand, the
monomial $T_{j+1}^{d-j}$ occurs in the dense representation of
$\Delta^jF(\boldsymbol A_0,T_1,\ldots,T_{j+1})$ with a nonzero
coefficient. We deduce that $\Delta^jF(\boldsymbol
A_0,T_1,\ldots,T_{j+1})$ cannot be a zero divisor modulo
$G_1,\ldots, G_m,\Delta^{i-1}F(\boldsymbol A_0,T_1,\ldots,T_i)$
$(1\le i\le j)$, which finishes the proof of our assertion. This
implies the statement of the lemma.
\end{proof}
%
%
\subsection{The dimension of the singular locus of $\Gamma_{r}^*$}
Next we show that the singular locus of $\Gamma_{r}^*$ has
codimension at least $2$ in $\Gamma_{r}^*$. For this purpose, we
carry out an analysis of the singular locus of $\Gamma_{r}^*$ which
generalizes that of \cite[Section 4.1]{MaPePr14} to this (eventually
nonlinear) setting. We start with the following criterion of
nonsingularity.
\begin{lemma}
\label{lemma: jacobian_F full rank implies nonsingular} Let
$J_{\boldsymbol G,\boldsymbol F}\in \cfq[\boldsymbol A_0,
\boldsymbol{T}]^{(m+r)\times (d+r)}$ be the Jacobian matrix of
$\boldsymbol G:=(G_1,\ldots,G_m)$ and $F(\boldsymbol A_0, T_i)$
$(1\leq i\leq r)$ with respect to $\boldsymbol A_0$, $\boldsymbol
T$, and let $(\boldsymbol a_0,\boldsymbol\alpha)\in \Gamma_{r}^*$.
If $J_{\boldsymbol G,\boldsymbol F}(\boldsymbol
a_0,\boldsymbol\alpha)$ has full rank, then $(\boldsymbol
a_0,\boldsymbol\alpha)$ is a nonsingular point of $\Gamma_{r}^*$.
\end{lemma}

\begin{proof}
Considering the Newton form of the polynomial interpolating
$F(\boldsymbol a_0,T)$ at $\alpha_1,\ldots,\alpha_r$ we easily
deduce that $F(\boldsymbol a_0,\alpha_i)=0$ for $1 \leq i \leq r$.
This shows that $F(\boldsymbol A_0,T_i)$ vanishes on $\Gamma_r^*$
for $1\le i\le r$. As a consequence, any element of the tangent
space $\mathcal{T}_{(\boldsymbol a_0,\boldsymbol\alpha)}\Gamma_r^*$
of $\Gamma_r^*$ at $(\boldsymbol a_0,\boldsymbol\alpha)$ belongs to
the kernel of the Jacobian matrix $J_{\boldsymbol G,\boldsymbol
F}(\boldsymbol a_0,\boldsymbol\alpha)$.

By hypothesis, the $(m+r)\times (d+r)$ matrix $J_{\boldsymbol
G,\boldsymbol F}(\boldsymbol a_0,\boldsymbol\alpha)$ has full rank
$m+r$, and thus its kernel has dimension $d-m$. We conclude that the
tangent space $\mathcal{T}_{(\boldsymbol
a_0,\boldsymbol\alpha)}\Gamma_r^*$ has dimension at most $d-m$.
Since $\Gamma_r^*$ is of pure dimension $d-m$, it follows that
$(\boldsymbol a_0,\boldsymbol\alpha)$ is a nonsingular point of
$\Gamma_r^*$.
\end{proof}

Let $(\boldsymbol a_0,\boldsymbol\alpha):=(\boldsymbol
a_0,\alpha_1,\ldots,\alpha_r)$ be an arbitrary point of $\Gamma_r^*$
and let $f_{\boldsymbol a_0}:=F(\boldsymbol a_0,T)$. Then the
Jacobian matrix $J_{\boldsymbol G,\boldsymbol F}(\boldsymbol
a_0,\boldsymbol\alpha)$ has the following form:
$$J_{\boldsymbol G,\boldsymbol F}(\boldsymbol a_0,\boldsymbol\alpha):=\left(\begin{array}{cc}
\dfrac{\partial \boldsymbol G}{\partial \boldsymbol A_0}(\boldsymbol a_0,\boldsymbol{\alpha}) & \boldsymbol0\\
&\\
\dfrac{\partial F}{\partial \boldsymbol A_0}(\boldsymbol
a_0,\boldsymbol{\alpha})& \dfrac{\partial F}{\partial \boldsymbol
T}(\boldsymbol a_0,\boldsymbol{\alpha})
\end{array}\right).$$
Observe that $(\partial F/\partial\boldsymbol T)(\boldsymbol
a_0,\boldsymbol\alpha)$ is a diagonal matrix whose $i$th diagonal
entry is $f_{\boldsymbol a_0}'(\alpha_i)$. If $(\partial \boldsymbol
G /
\partial \boldsymbol A_0)(\boldsymbol a_0,\boldsymbol{\alpha})$ has full
rank and all the roots in $\cfq$ of $f_{\boldsymbol a_0}$ are
simple, then $J_{\boldsymbol G,\boldsymbol F}(\boldsymbol
a_0,\boldsymbol\alpha)$ has full rank and $(\boldsymbol
a_0,\boldsymbol\alpha)$ is a regular point of $\Gamma_r^*$.
Therefore, to prove that the singular locus of $\Gamma_r^*$ is a
subvariety of codimension at least $2$ in $\Gamma_r^*$, it suffices
to consider the set of points $(\boldsymbol a_0,\boldsymbol\alpha)
\in\Gamma_r^*$ for which either $\boldsymbol a_0$ is a singular
point of $V$, or at least one coordinate of $\boldsymbol{\alpha}$ is
a multiple root of $f_{\boldsymbol a_0}$. To deal with the first of
these cases, consider the morphism of $\fq$--varieties defined as
follows:
\begin{equation}\label{eq: morfismo finito - mean}
\begin{array}{rrcl}
\Psi_r:& {\Gamma}_{r}^{*}& \to& V\\
   &(\boldsymbol a_0,\boldsymbol\alpha)& \mapsto&\boldsymbol a_0,
\end{array}
\end{equation}
We have the following result.
\begin{lemma}\label{lemma: Psi_r is finite}
$\Psi_r$ is a finite dominant morphism.
\end{lemma}
\begin{proof}
It is easy to see that $\Psi_r$ is a surjective mapping. Therefore,
it suffices to show that the coordinate function $t_i$ of
$\cfq[\Gamma_r^*]$ defined by $T_i$ satisfies a monic equation with
coefficients in $\cfq[V]$ for $1\le i\le r$. Denote by $\xi_j$ the
coordinate function of $V$ defined by $A_j$ for $0\le j\le d-1$, and
set $\boldsymbol\xi_0:=(\xi_{d-1},\ldots,\xi_0)$. Since the
polynomial $F(\boldsymbol{A}_0,T_i)$ vanishes on $\Gamma_r^*$ for
$1\le i\le r$, and is a monic element of $\cfq[\boldsymbol
A_0][T_i]$, we deduce that the monic element
$F(\boldsymbol\xi_0,T_i)$ of $\cfq[V][T_i]$ annihilates $t_i$ in
$\Gamma_r^*$ for $1\le i\le r$. This shows that the ring extension
$\cfq[V]\hookrightarrow \cfq[\Gamma_r^*]$ is integral, namely
$\Psi_r$ is a finite dominant mapping.
\end{proof}

A first consequence of Lemma \ref{lemma: Psi_r is finite} is that
the set of points $(\boldsymbol
a_0,\boldsymbol\alpha)\in\Gamma_{r}^*$ with $\boldsymbol a_0$
singular are under control.
\begin{corollary}\label{coro: incidence var singular points V in codim 2}
The set $\mathcal{W}_0$ of points $(\boldsymbol
a_0,\boldsymbol\alpha)\in\Gamma_{r}^*$ with $\boldsymbol
a_0\in\mathrm{Sing}(V)$ is contained in a subvariety of codimension
2 of $\Gamma_r^*$.
\end{corollary}
\begin{proof}
Hypotheses $({\sf H_1})$ and $({\sf H_2})$ imply that $V$ is a
normal complete intersection. It follows that $\mathrm{Sing}(V)$ has
codimension at least two in $V$. Then
$\mathcal{W}_0=\Psi_r^{-1}(\mathrm{Sing}(V))$ has codimension at
least two in $\Gamma_r^*$.
\end{proof}

By Corollary \ref{coro: incidence var singular points V in codim 2}
it suffices to consider the set of singular points $(\boldsymbol
a_0,\boldsymbol\alpha)$ of $\Gamma_r^*$ with $\boldsymbol a_0\in
V\setminus\mathrm{Sing}(V)$. By the remarks before Lemma \ref{lemma:
Psi_r is finite}, if $(\boldsymbol a_0,\boldsymbol\alpha)$ is such a
singular point, then $f_{\boldsymbol a_0}$ must have multiple roots.
We start considering the ``extreme'' case where $f'_{\boldsymbol
a_0}$ is the zero polynomial.
\begin{lemma}\label{lemma: f'=0}
The set $\mathcal{W}_1$ of  points $(\boldsymbol
a_0,\boldsymbol\alpha)\in\Gamma_{r}^*$ with $f'_{\boldsymbol a_0}=0$
is contained in a subvariety of codimension 2 of $\Gamma_r^*$.
\end{lemma}

\begin{proof}
The condition $f'_{\boldsymbol a_0}=0$ implies $\boldsymbol a_0$
belongs to both the discriminant locus $\mathcal{D}(V)$ and the
first subdiscriminant locus $\mathcal{S}_1(V)$, namely
$\mathcal{W}_1\subset \Psi_r^{-1}(\mathcal{D}(V)
\cap\mathcal{S}_1(V))$. Hypothesis $({\sf H_4})$ asserts that
$\mathcal{D}(V)\cap\mathcal{S}_1(V)$ has codimension two in $V$.
Therefore, taking into account that $\Psi_r$ is a finite morphism we
deduce that $\mathcal{W}_1$ has codimension two in $\Gamma_r^*$.
\end{proof}

In what follows we shall assume that $\boldsymbol a_0$ is a regular
point of $V$, $f'_{\boldsymbol a_0}$ is nonzero and $f_{\boldsymbol
a_0}$ has multiple roots. We analyze the case where exactly one of
the coordinates of $\boldsymbol{\alpha}$ is a multiple root of
$f_{\boldsymbol a_0}$.

\begin{lemma}\label{lemma: only one multiple root - mean}
Suppose that there exists a unique coordinate $\alpha_i$ of
$\boldsymbol{\alpha}$ which is a multiple root of $f_{\boldsymbol
a_0}$. Then $(\boldsymbol a_0,\boldsymbol{\alpha})$ is a regular
point of $\Gamma_r^{*}$.
\end{lemma}
\begin{proof}
Assume without loss of generality that $\alpha_1$ is the only
multiple root of $f_{\boldsymbol a_0}$ among the coordinates of
$\boldsymbol{\alpha}$. According to Lemma \ref{lemma: jacobian_F
full rank implies nonsingular}, it suffices to show that the
Jacobian matrix $J_{\boldsymbol G,\boldsymbol F}(\boldsymbol
a_0,\boldsymbol{\alpha})$ has full rank. For this purpose, consider
the $r\times (r+1)$--submatrix $\partial F/\partial(A_0,\boldsymbol
T)(\boldsymbol a_0,\boldsymbol{\alpha})$ of $J_{\boldsymbol
G,\boldsymbol F}(\boldsymbol a_0,\boldsymbol{\alpha})$ consisting of
the entries of its last $r$ rows and its last $r+1$ columns:
$$\frac{\partial F}{\partial(A_0,\boldsymbol T)}(\boldsymbol a_0,\boldsymbol{\alpha}):=\left(\begin{array}{ccccccccc}
 1          &0      & 0& 0                      & \cdots & 0\\
 1          &0      & f_{\boldsymbol a_0}'(\alpha_2)& 0                      & \cdots & 0\\
\vdots  & \vdots  & 0   & \ddots       & \ddots   &\vdots \\
\vdots  &  \vdots & \vdots   &  \ddots      &  \ddots  &0 \\
 1 & 0 & 0 & \cdots& 0& f_{\boldsymbol a_0}'(\alpha_r)
\end{array}\right).
$$
Since by hypothesis $\alpha_i$ is a simple root of $f'_{\boldsymbol
a_0}$ for $i\ge 2$, we have $f'_{\boldsymbol a_0}(\alpha_i)\neq 0$
for $i\ge 2$, and thus $\partial F/\partial(A_0,\boldsymbol
T)(\boldsymbol a_0,\boldsymbol{\alpha})$ is of full rank $r$.

On the other hand, since the matrix $(\partial \boldsymbol
G/\partial{\boldsymbol A_0})(\boldsymbol a_0,\boldsymbol{\alpha})$
has rank $m$ and its last column is zero, denoting by $(\partial
\boldsymbol G/\partial{\boldsymbol A})(\boldsymbol
a_0,\boldsymbol{\alpha})$ the submatrix of $(\partial \boldsymbol
G/\partial{\boldsymbol A_0})(\boldsymbol a_0,\boldsymbol{\alpha})$
obtained by deleting its last column, we see that $J_{\boldsymbol
G,\boldsymbol F}(\boldsymbol a_0,\boldsymbol{\alpha})$ can be
expressed as the following block matrix:
$$J_{\boldsymbol G,\boldsymbol F}(\boldsymbol a_0,\boldsymbol{\alpha})
=\left(\begin{array}{cc}
\dfrac{\partial \boldsymbol G}{\partial \boldsymbol A}(\boldsymbol a_0,\boldsymbol{\alpha}) & \boldsymbol0\\
&\\
\ast& \dfrac{\partial F}{\partial(A_0,\boldsymbol T)}(\boldsymbol
a_0,\boldsymbol{\alpha})
\end{array}\right).
$$
Since both $(\partial \boldsymbol G/\partial \boldsymbol
A)(\boldsymbol a_0,\boldsymbol{\alpha})$ and $(\partial
F/\partial(A_0,\boldsymbol T))(\boldsymbol a_0,\boldsymbol{\alpha})$
have full rank, we conclude that $J_{\boldsymbol G,\boldsymbol
F}(\boldsymbol a_0,\boldsymbol{\alpha})$ has rank $m+r$.
\end{proof}
Now we analyze the case where two distinct multiple roots of
$f_{\boldsymbol a_0}$ occur among the coordinates of
$\boldsymbol\alpha$.
\begin{lemma} \label{lemma: two distinct multiple roots - mean}
Let $\mathcal{W}_2$ be the set of points $(\boldsymbol
a_0,\boldsymbol{\alpha})\in \Gamma_{r}^{*}$ for which there exist
$1\le i<j\le r$ such that $\alpha_i\not=\alpha_j$ and
$\alpha_i,\alpha_j$ are multiple roots of $f_{\boldsymbol a_0}$.
Then $\mathcal{W}_2$ is contained in a subvariety of codimension 2
of $\Gamma_{r}^{*}$.
\end{lemma}
\begin{proof}
Let $(\boldsymbol a_0,\boldsymbol{\alpha})$ be an arbitrary point of
$\mathcal{W}_2$. Since $f_{\boldsymbol a_0}$ has at least two
distinct multiple roots, the greatest common divisor of
$f_{\boldsymbol a_0}$ and $f'_{\boldsymbol a_0}$ has degree at least
$2$. This implies that
$$\mathrm{Disc}(f_{\boldsymbol a_0})=
\mathrm{Subdisc}(f_{\boldsymbol a_0})=0.$$
As a consequence, $\mathcal{W}_2\subset\Psi_r^{-1}(\mathcal{Z})$,
where $\Psi_r$ is the morphism of (\ref{eq: morfismo finito -
mean}), $\mathcal{Z}:=\mathcal{D}(V)\cap\mathcal{S}_1(V)$ and
$\mathcal{D}(V)$ and $\mathcal{S}_1(V)$ are the discriminant locus
and the first subdiscriminant locus of $V$. Hypothesis $({\sf H_4})$
proves that $\mathcal{Z}$ has codimension two in $V$. It follows
that $\dim \mathcal{Z}=d-m-2$, and hence $\dim
\Psi_r^{-1}(\mathcal{Z})=d-m-2$. The statement of the lemma follows.
\end{proof}

Next we consider the case where only one multiple root of
$f_{\boldsymbol a_0}$ occurs among the coordinates of
$\boldsymbol{\alpha}$, and at least two distinct coordinates of
$\boldsymbol{\alpha}$ take this value. Then we have either that all
the remaining coordinates of $\boldsymbol{\alpha}$ are simple roots
of $f_{\boldsymbol a_0}$, or there exists at least a third
coordinate whose value is the same multiple root. We now deal with
the first of these two cases.
\begin{lemma}\label{lemma: repeated multiple root I - mean} Let
$(\boldsymbol a_0, \boldsymbol\alpha)\in \Gamma_{r}^*$ be a point
satisfying the conditions:
\begin{itemize}
\item there exist $1\le i<j\le r$ such that $\alpha_i=\alpha_j$ and
$\alpha_i$ is a multiple root of $f_{\boldsymbol a_0}$;
\item for any $k\notin\{i,j\}$, $\alpha_k$ is a simple root of
$f_{\boldsymbol a_0}$.
\end{itemize}
Then either $(\boldsymbol a_0, \boldsymbol\alpha)$ is regular point
of $\Gamma_{r}^*$ or it is contained in a subvariety $\mathcal{W}_3$
of codimension two of $\Gamma_r^*$.
\end{lemma}
\begin{proof}
We may assume without loss of generality that $i=1$ and $j=2$.
Observe that $\Delta^1F(\boldsymbol A_0,T_1,T_2)$ and $F(\boldsymbol
A_0,T_i)$ $(2\le i\le r)$ vanish on $\Gamma_r^*$. Therefore, the
tangent space $\mathcal{T}_{(\boldsymbol
a_0,\boldsymbol\alpha)}\Gamma_r^*$ of $\Gamma_r^*$ at $(\boldsymbol
a_0,\boldsymbol\alpha)$ is included in the kernel of the Jacobian
matrix $J_{\boldsymbol G,\Delta^1,\boldsymbol F^*}(\boldsymbol
a_0,\boldsymbol\alpha)$ of $\boldsymbol G$, $\Delta^1 F(\boldsymbol
A_0,T_1,T_2)$ and $F(\boldsymbol A_0,T_i)$ $(2\le i\le r)$ with
respect to $\boldsymbol A_0,\boldsymbol{T}$. We claim that either
$J_{\boldsymbol G,\Delta^1,\boldsymbol F^*}(\boldsymbol
a_0,\boldsymbol\alpha)$ has rank $r+m$, or $(\boldsymbol
a_0,\boldsymbol \alpha)$ is contained in a subvariety of codimension
two of $\Gamma_r^*$.

Now we prove the claim. We may express $J_{\boldsymbol
G,\Delta^1,\boldsymbol F^*}(\boldsymbol a_0,\boldsymbol\alpha)$ as
$$
J_{\boldsymbol G,\Delta^1,\boldsymbol F^*}
(\boldsymbol a_0,\boldsymbol\alpha)=
\left(\begin{array}{cc} \dfrac{\partial \boldsymbol G}{\partial
\boldsymbol A}
(\boldsymbol a_0,\boldsymbol{\alpha}) & \boldsymbol0\\
&\\ \ast& \dfrac{\partial (\Delta^1,\boldsymbol
F^*)}{\partial(A_0,\boldsymbol T)}(\boldsymbol
a_0,\boldsymbol\alpha)
\end{array}\right),
$$
where $(\partial \boldsymbol G/ \partial \boldsymbol A)(\boldsymbol
a_0,\boldsymbol{\alpha})\in \fq^{m\times (d-1)}$ is the Jacobian
matrix of $\boldsymbol G$ with respect to $\boldsymbol A$ and
$(\partial (\Delta^1,\boldsymbol F^*)/\partial(A_0,\boldsymbol
T))(\boldsymbol a_0,\boldsymbol\alpha)\in \fq{}^{r\times (r+1)}$ is
the Jacobian matrix of $\Delta^1 F(\boldsymbol A_0,T_1,T_2)$ and
$F(\boldsymbol A_0,T_i)$ $(2\le i\le r)$ with respect to
$A_0,\boldsymbol{T}$:
$$\frac{\partial
(\Delta^1,\boldsymbol F^*)}{\partial(A_0,\boldsymbol T)}(\boldsymbol
a_0,\boldsymbol\alpha):=\left(\begin{array}{ccccccccc}
0& \lambda_1  & \lambda_2 & 0                    & \cdots & 0\\
1& 0 & 0 & 0                     & \cdots & 0\\
1 & 0 & 0  & f_{\boldsymbol a_0}'(\alpha_3)       & \ddots   &\vdots \\
\vdots   &  \vdots  &  \vdots  &  \vdots  &  \ddots  &0 \\
1 &0 & 0&  0&\cdots & f_{\boldsymbol a_0}'(\alpha_r)
\end{array}\right).
$$

Now we determine $\lambda_i:=(\partial \Delta^1F(\boldsymbol
A_0,T_1,T_2)/\partial T_i)(\boldsymbol a_0,\boldsymbol\alpha)$ for
$i=1,2$. Observe that
$$\frac{\partial}{\partial T_2}\Bigg(\frac{T_2^j-T_1^j}{T_2-T_1}\Bigg)=
\frac{jT_2^{j-1}(T_2-T_1)-(T_2^j-T_1^j)}{(T_2-T_1)^2}=\sum_{k=2}^j
\binom{j}{k}T_2^{j-k}(T_1-T_2)^{k-2}.$$
It follows that
$$\lambda_2:=\frac{\partial \Delta^1F(\boldsymbol
A_0,T_1,T_2)}{\partial T_2}(\boldsymbol a_0,\boldsymbol\alpha)=
\sum_{j=2}^da_j\frac{j(j-1)}{2}\alpha_1^{j-2}=\Delta^2f_{\boldsymbol
a_0}(\alpha_1,\alpha_1).$$
Furthermore, it is easy to see that $\lambda_1=-\lambda_2$.

Recall that $\alpha_i$ is a simple root of $f_{\boldsymbol a_0}$ for
$i\ge 3$, which implies $f_{\boldsymbol a_0}'(\alpha_i)\not=0$ for
$i\ge 3$. If $\Delta^2f_{\boldsymbol a_0}(\alpha_1,\alpha_1)\not=0$,
then there exists an $(r\times r)$--submatrix of $(\partial
(\Delta^1,\boldsymbol F^*)/\partial(A_0,\boldsymbol T)) (\boldsymbol
a_0,\boldsymbol\alpha)$ with rank $r$. Thus, $J_{\boldsymbol
G,\Delta^1,\boldsymbol F^*}(\boldsymbol a_0,\boldsymbol\alpha)$ has
rank $r+m$, namely the first assertion of the claim holds. It
follows that the kernel of $J_{\boldsymbol G,\Delta^1,\boldsymbol
F^*}(\boldsymbol a_0,\boldsymbol\alpha)$ has dimension $d-m$. This
implies that $\dim \mathcal{T}_{(\boldsymbol a_0,
\boldsymbol\alpha)}\Gamma_r^* \le d-m$, which proves that
$(\boldsymbol a_0, \boldsymbol\alpha)$ is regular point of
$\Gamma_r^*$.

On the other hand, for a point $(\boldsymbol a_0,
\boldsymbol\alpha)\in \Gamma_{r}^*$ as in the statement of the lemma
with $\Delta^2f_{\boldsymbol a_0}(\alpha_1,\alpha_1)=0$, we have
that $\alpha_1$ is root of multiplicity at least three of
$f_{\boldsymbol a_0}$. As a consequence, the set $\mathcal{W}_3$ of
such points is contained in
$\Psi_r^{-1}\big(\mathcal{D}(V)\cap\mathcal{S}_1(V)\big)$. The lemma
follows arguing as in the proof of Lemma \ref{lemma: two distinct
multiple roots - mean}.
\end{proof}

Finally, we analyze the set of points of $\Gamma_r^*$ such that the
value of at least three distinct coordinates of
$\boldsymbol{\alpha}$ is the same multiple root of $f_{\boldsymbol
a_0}$.
\begin{lemma}\label{lemma: repeated multiple root II - mean}
Let $\mathcal{W}_4\subset\Gamma_{r}^*$ be the set of points
$(\boldsymbol a_0, \boldsymbol\alpha)$ for which there exist $1\le
i<j<k\le r$ such that $\alpha_i=\alpha_j=\alpha_k$ and $\alpha_i$ is
a multiple root of $f_{\boldsymbol a_0}$. Then $\mathcal{W}_4$ is
contained in a subvariety of codimension $2$ in $\Gamma_r^*$.
\end{lemma}
\begin{proof}
Let $(\boldsymbol a_0, \boldsymbol\alpha)$ be an arbitrary point of
$\mathcal{W}_4$. Without loss of generality we may assume that
$\alpha_1=\alpha_2=\alpha_3$ is the multiple root of $f_{\boldsymbol
a_0}$ of the statement of the lemma. Since $(\boldsymbol a_0,
\boldsymbol\alpha)$ satisfies the equations
$$
F(\boldsymbol A_0,T_1)=\Delta F(\boldsymbol A_0,T_1,T_2)=
\Delta^{2}F(\boldsymbol A_0,T_1,T_2,T_3)=0,
$$
we see that $\alpha_1$ is a common root of $f_{\boldsymbol a_0}$,
$\Delta F(\boldsymbol a_0,T,T)$ and $\Delta^2F(\boldsymbol
a_0,T,T,T)$. It follows that $\alpha_1$ is a root of multiplicity at
least 3 of $f_{\boldsymbol a_0}$, and thus the greatest common
divisor of $f_{\boldsymbol a_0}$ and $f'_{\boldsymbol a_0}$ has
degree at least $2$. As a consequence, the proof follows by the
arguments of the proof of Lemma \ref{lemma: two distinct multiple
roots - mean}.
\end{proof}

Now we can prove the main result of this section. According to
Lemmas \ref{lemma: f'=0}, \ref{lemma: only one multiple root -
mean}, \ref{lemma: two distinct multiple roots - mean}, \ref{lemma:
repeated multiple root I - mean} and \ref{lemma: repeated multiple
root II - mean}, the set of singular points of $\Gamma_r^*$ is
contained in the set $\mathcal{W}_1\cup
\mathcal{W}_2\cup\mathcal{W}_3\cup\mathcal{W}_4$, where
$\mathcal{W}_1$, $\mathcal{W}_2$, $\mathcal{W}_3$ and
$\mathcal{W}_4$ are defined in the statement of Lemmas \ref{lemma:
f'=0}, \ref{lemma: two distinct multiple roots - mean}, \ref{lemma:
repeated multiple root I - mean} and \ref{lemma: repeated multiple
root II - mean}. Since each set $\mathcal{W}_i$ is contained in a
subvariety of codimension $2$ of $\Gamma_r^*$, we obtain the
following result.
\begin{theorem}\label{th: codimension singular locus Gamma j}
Let $q>d\ge m+2$. The singular locus of $\Gamma_r^*$ has codimension
at least $2$ in $\Gamma_r^*$.
\end{theorem}

In the next result we deduce important consequences of Theorem
\ref{th: codimension singular locus Gamma j}.
\begin{corollary}\label{coro: J is radical}
With assumptions as in Theorem \ref{th: codimension singular locus
Gamma j}, the ideal $J\subset \fq[\boldsymbol A_0,\boldsymbol{T}]$
generated by $G_1,\ldots,G_m$ and $\Delta^{i-1}F(\boldsymbol
A_0,T_1,\ldots,T_i)$ $(1\leq i\leq r)$ is radical. Moreover,
$\Gamma_r^{*}$ is an ideal--theoretic complete intersection of
dimension $d-m$.
\end{corollary}
\begin{proof}
We prove that $J$ is a radical ideal. Denote by $J_{\boldsymbol
G,\boldsymbol \Delta}(\boldsymbol{A}_0,\boldsymbol{T})$ the Jacobian
matrix of $G_1,\ldots,G_m$ and $\Delta^{i-1}F({\boldsymbol
A_0,T_1,\ldots,T_i})$ ($1\leq i\leq r$) with respect to
$\boldsymbol{A}_0,\boldsymbol{T}$. By Lemma \ref{lemma: gamma j is
set-theoretic complete intersection}, these polynomials form a
regular sequence. Hence, according to \cite[Theorem
18.15]{Eisenbud95}, it is sufficient to prove that the set of points
$(\boldsymbol{a}_0,\boldsymbol{\alpha}) \in\Gamma_r^*$ for which
$J_{\boldsymbol G,\boldsymbol
\Delta}(\boldsymbol{a}_0,\boldsymbol{\alpha})$ does not have full
rank is contained in a subvariety of $\Gamma_r^*$ of codimension at
least 1.

In the proof of Lemma \ref{lemma: jacobian_F full rank implies
nonsingular} we show that $F(\boldsymbol A_0,T_i)\in J$ for $1\leq i
\leq r$. This implies that each gradient $\nabla
F(\boldsymbol{a}_0,\alpha_i)$ is a linear combination of the
gradients of $G_1,\ldots,G_m$ and $\Delta^{i-1}F(\boldsymbol
A_0,T_1,\ldots,T_i)$ $(1\leq i\leq r)$. We conclude that
$\mathrm{rank}\,J_{\boldsymbol G,\boldsymbol
F}(\boldsymbol{a}_0,\boldsymbol{\alpha})\le
\mathrm{rank}\,J_{\boldsymbol G,\boldsymbol
\Delta}(\boldsymbol{a}_0,\boldsymbol{\alpha})$.

Let $(\boldsymbol{a}_0,\boldsymbol{\alpha})$ be an arbitrary point
of $\Gamma_r^*$ such that $J_{\boldsymbol G,\boldsymbol
\Delta}(\boldsymbol{a}_0,\boldsymbol{\alpha})$ does not have full
rank. By Corollary \ref{coro: incidence var singular points V in
codim 2} we may assume without loss of generality that $\boldsymbol
a_0\in V\setminus\mathrm{Sing}(V)$. Then $J_{\boldsymbol
G,\boldsymbol F}(\boldsymbol{a}_0,\boldsymbol{\alpha})$ does not
have full rank and thus $f_{\boldsymbol a_0}$ has multiple roots.
Observe that the set of points
$(\boldsymbol{a}_0,\boldsymbol{\alpha})\in\Gamma_r^*$ for which
$f_{\boldsymbol a_0}$ has multiple roots is equal to
$\Psi_r^{-1}(\mathcal{D}(V))$, where $\mathcal{D}(V)$ is
discriminant locus of $V$. According to hypothesis $({\sf H_4})$,
$\mathcal{D}(V)$ has codimension one in $V$, which implies that
$\Psi_r^{-1}(\mathcal{D}(V))$ has codimension 1 of $\Gamma_r^*$. It
follows that the set of points
$(\boldsymbol{a}_0,\boldsymbol{\alpha}) \in\Gamma_r^*$ for which
$J_{\boldsymbol G,\boldsymbol
\Delta}(\boldsymbol{a}_0,\boldsymbol{\alpha})$ does not have full
rank is contained in a subvariety of $\Gamma_r^*$ of codimension at
least 1. Hence, $J$ is a radical ideal, which in turn implies that
$\Gamma_r^*$ is an ideal--theoretic complete intersection of
dimension $d-m$.
\end{proof}
%
%
\section{The geometry of the projective closure of $\Gamma_r^*$}
\label{section: geometry of pcl Gamma r}
To estimate the number of $\fq$--rational points of $\Gamma_{r}^*$
we need information on the behavior of $\Gamma_{r}^*$ at infinity.
For this purpose, we shall analyze the projective closure of
$\Gamma_{r}^*$, whose definition we now recall. Consider the
embedding of $\A^{d+r}$ into the projective space $\Pp^{d+r}$ which
assigns to any point $(\boldsymbol a_0,\boldsymbol{\alpha})\in
\A^{d+r}$ the point
$(a_{d-1}:\dots:a_0:1:\alpha_1:\dots:\alpha_r)\in \Pp^{d+r}$. The
closure in the Zariski topology of $\Pp^{d+r}$ of the image of
$\Gamma_{r}^*$ under this embedding is called the {\em projective
closure} $\mathrm{pcl}(\Gamma_{r}^*)\subset\Pp^{d+r}$ of
$\Gamma_{r}^*$. The points of $\mathrm{pcl}(\Gamma_{r}^*)$ lying in
$\{T_0=0\}$ are called the points of $\mathrm{pcl}(\Gamma_{r}^*)$
{\em at infinity}.

It is well--known that $\mathrm{pcl}(\Gamma_{r}^*)$ is the
$\fq$--variety of $\Pp^{d+r}$ defined by the homogenization $F^h \in
\fq[\boldsymbol{A}_0,T_0,\boldsymbol{T}]$ of each polynomial $F$ in
the ideal $J\subset \fq[\boldsymbol A_0,\boldsymbol T]$ generated by
$G_1,\ldots,G_m$ and $\Delta^{i-1}F(\boldsymbol A_0,T_1,\ldots,T_i)$
$(1\leq i \leq r)$. We denote by $J^{h}$ the ideal generated by all
the polynomials $F^h$ with $F\in J$. Since $J$ is radical it turns
out that $J^h$ is also radical (see, e.g., \cite[\S I.5, Exercise
6]{Kunz85}). Furthermore, $\mathrm{pcl}(\Gamma_{r}^*)$ is of pure
dimension $d-m$ (see, e.g., \cite[Propositions I.5.17 and
II.4.1]{Kunz85}) and degree equal to $\deg\Gamma_r^*$ (see, e.g.,
\cite[Proposition 1.11]{CaGaHe91}).

\begin{lemma}\label{lemma: pcl gamma j is set-theoretic complete int}
The polynomials $G^h_1,\ldots,G^h_m$ and $\Delta^{i-1}F(\boldsymbol
A_0,T_1,\ldots,T_i)^h$ $(1\leq i \leq r)$ form a regular sequence of
$\fq[\boldsymbol{A}_0,T_0,\boldsymbol{T}]$.
\end{lemma}
\begin{proof}
We claim $G^h_1,\ldots,G^h_m$ form a regular sequence of
$\fq[\boldsymbol{A}_0,T_0,\boldsymbol{T}]$. Indeed, the projective
subvariety of $\Pp^{d+r}$ defined by $T_0$ and $G^h_1,\ldots,G^h_m$
is isomorphic to the subvariety of $\Pp^{d+r-1}$ defined by
$G^{d_1}_1,\ldots,G^{d_m}_m$. Since hypothesis $({\sf H_3})$ implies
that the latter is of pure dimension $d+r-m-1$, we conclude that the
subvariety of $\Pp^{d+r}$ defined by $G^h_1,\ldots,G^h_m$ is of pure
dimension $d+r-m$. It follows that $G^h_1,\ldots,G^h_m$ form a
regular sequence. Then the lemma follows arguing as in the proof of
Lemma \ref{lemma: gamma j is set-theoretic complete intersection}.
\end{proof}

\begin{proposition}\label{prop: pcl gamma j is ideal-theoretic complete inters}
The projective variety defined by $G^h_1,\ldots,G^h_m$ and
$\Delta^{i-1}F(\boldsymbol A_0,T_1,\ldots,T_i)^h$ $(1\leq i \leq r)$
is $\mathrm{pcl}(\Gamma_r^*)$. Therefore,
$\mathrm{pcl}(\Gamma_{r}^*)$ is a set--theoretic complete
intersection of dimension $d-m$.
\end{proposition}
\begin{proof}
By the Newton form of the polynomial interpolating $F(\boldsymbol
A_0,T)$ at the points $T_1,\ldots, T_r$ we conclude that
$$F(\boldsymbol A_0,T_j)=\sum_{i=1}^r\Delta^{i-1}F(\boldsymbol
A_0,T_1,\ldots,T_i)\,(T_j-T_1)\cdots(T_j-T_{i-1}).$$
Homogenizing both sides of this equality we deduce that
$F(\boldsymbol A_0,T_j)^h$ belongs to the ideal of $\fq[\boldsymbol
A_0,T_0,\boldsymbol T]$ generated by $G^h_1,\ldots,G^h_m$ and
$\Delta^{i-1}F(\boldsymbol A_0,T_1,\ldots,T_i)^h$. Denote by
$V^h\subset\Pp^{d+r}$ the variety defined by all these polynomials.
Any point of $V^h\cap\{T_0=0\}\subset\Pp^{d+r}$ satisfies
\begin{align*}
F(\boldsymbol
A_0,T_i)^h|_{T_0=0}=T_i^{d}+A_{d-1}T_i^{d-1}=T_i^{d-1}(T_i+A_{d-1})&=0\quad(1\le
i\le r),\\
G^h_k(\boldsymbol A,T_0)^h|_{T_0=0}=G^{d_k}_k(\boldsymbol
A)&=0\quad(1\le k\le m).\nonumber
\end{align*}
We deduce that $V^h\cap\{T_0=0\}$ is contained in the union
$$V^h\cap\{T_0=0\}\subset \bigcup_{\mathcal{I}\subset\{1,\ldots, r\}}^r V_\mathcal{I}\cap\{T_0=0\},$$
where $V_\mathcal{I}\subset\Pp^{d+r}$ is the variety defined by
$T_i=0$ ($i \in\mathcal{I}$), $T_j+A_{d-1}=0$ $(j\in\{1,\ldots,
r\}\setminus\mathcal{I})$ and $G_k^{d_k}=0$ ($1\leq k \leq m$). As
any $V_\mathcal{I}\cap \{T_0=0\}$ is of pure dimension $d-m-1$,
$V^h\cap \{T_0=0\}$ is of pure dimension $d-m-1$.

Lemma \ref{lemma: pcl gamma j is set-theoretic complete int} implies
that $V^h$ is of pure dimension $d-m$, and thus it has no
irreducible component in the hyperplane at infinity. In particular,
it agrees with the projective closure of its restriction to
$\A^{d+r}$ (see, e.g., \cite[Proposition I.5.17]{Kunz85}). As this
restriction is $\Gamma_r^*$, we have $V^h=\mathrm{pcl}(\Gamma_r^*)$.
\end{proof}
%
%
\subsection{The singular locus of $\mathrm{pcl}(\Gamma_{r}^*)$}
Next we study the singular locus of $\mathrm{pcl}(\Gamma_{r}^*)$. We
start with the following characterization of the points of
$\mathrm{pcl}(\Gamma_{r}^*)$ at infinity.
\begin{lemma}\label{lemma: singular locus pcl Gamma_j at infinity}
$\mathrm{pcl}(\Gamma_{r}^*)\cap \{T_0=0\}\subset \Pp^{d+r-1}$ is
contained in a union of $r+1$ normal complete intersections defined
over $\fq$, each of pure dimension $d-m-1$ and degree
$\prod_{i=1}^md_i$.
\end{lemma}
\begin{proof}
We claim that $\Delta^1F(\boldsymbol A_0,T_i,T_j)^h\in J^h$ for
$1\leq i<j\leq r$. Indeed, we have the identity
$\Delta^1F(\boldsymbol A_0,T_i,T_j)(T_i-T_j)=F(\boldsymbol
A_0,T_i)-F(\boldsymbol A_0, T_j)$. Since $F(\boldsymbol A_0, T_k)$
vanishes in $\Gamma^{*}_r$ for $1\leq k \leq r$, we deduce that
$\Delta^1F(\boldsymbol A_0,T_i,T_j)$ vanishes on the nonempty
Zariski open dense subset $\{T_i\not= T_j\}\cap\Gamma_r^*$ of
$\Gamma_r^*$. This implies that $\Delta^1F(\boldsymbol A_0,T_i,T_j)$
vanishes in $\Gamma_r^*$, which proves the claim.

Combining the claim with the fact that $F(\boldsymbol A_0,T_i)^h\in
J^h$ for $1\leq i \leq r$, we conclude that any $(\boldsymbol
a_0,\boldsymbol\alpha)\in \mathrm{pcl}(\Gamma_{r}^*)\cap \{T_0=0\}$
satisfies the following identities for $1\le i\le r$ and $1\le
i<j\le r$ respectively:
\begin{align}\label{eq: equations pcl at infinity I}
F(\boldsymbol A_0,T_i)^h|_{T_0=0}&=T_i^{d}+A_{d-1}T_i^{d-1}=T_i^{d-1}(T_i+A_{d-1})=0,\\
\Delta^{1}F(\boldsymbol A_0,T_i,T_j)^h|_{T_0=0}&=
\frac{T_i^{d}-T_j^{d}}{T_i-T_j}+A_{d-1}\frac{T_i^{d-1}-T_j^{d-1}}{T_i-T_j}\nonumber\\
&=\sum_{k=0}^{d-2}T_j^kT_i^{d-2-k}(T_i+A_{d-1})+T_j^{d-1}=0.
\label{eq: equations pcl at infinity II}
\end{align}
From \eqref{eq: equations pcl at infinity I}--\eqref{eq: equations
pcl at infinity II} we deduce that
$\mathrm{pcl}(\Gamma_{r}^*)\cap\{T_0=0\}$ is contained in a finite
union of $r+1$ normal complete intersections of $\Pp^{d+r-1}$
defined over $\fq$ of pure dimension $d-m-1$. More precisely, it can
be seen that
$$\mathrm{pcl}(\Gamma_{r}^*)\cap\{T_0=0\}\subset \bigcup_{i=0}^r
V_i\cap\{T_0=0\},$$
where $V_0$ is the variety defined by $T_i=0$ ($1\leq i \leq r$) and
$G_k^{d_k}=0$ ($1\leq k \leq m$), and $V_i$ $(1\leq i \leq r)$ is
defined as the set of solutions of
$$T_i+A_{d-1}=0,\ \ T_j=0\ \,(1\le j\le r,\,i\neq j), \ \ G_k^{d_k}=0\
\,(1\leq k \leq m).
$$

By Proposition \ref{prop: pcl gamma j is ideal-theoretic complete
inters} we have that $\mathrm{pcl}(\Gamma_{r}^*)$ is of pure
dimension $d-m$. Then each irreducible component $\mathcal{C}$ of
$\mathrm{pcl}(\Gamma_{r}^*)\cap \{T_0=0\}$ has dimension at least
$d-m-1$, and is contained in an irreducible component of a variety
$V_i$ for some $i\in\{0,\ldots, r\}$. By, e.g., \cite[\S 6.1,
Theorem 1]{Shafarevich94}, $\mathcal{C}$ is an irreducible component
of a variety $V_i$, finishing thus the proof of the lemma.
\end{proof}

Now we are able to upper bound the dimension of the  singular locus
of $\mathrm{pcl}(\Gamma_{r}^*)$ at infinity.
\begin{lemma}\label{lemma: dimension of singular locus at infinity}
The singular locus of $\mathrm{pcl}(\Gamma_{r}^*)$ at infinity has
dimension at most $d-m-2$.
\end{lemma}
\begin{proof}
By \cite[Lemma 1.1]{GhLa02a}, the singular locus of
$\mathrm{pcl}(\Gamma_r^*)$ at infinity is contained in the singular
locus of $\mathrm{pcl}(\Gamma_r^*)\cap \{T_0=0\}$. Lemma \ref{lemma:
singular locus pcl Gamma_j at infinity} proves that
$\mathrm{pcl}(\Gamma_r^*)\cap \{T_0=0\}$ has pure dimension $d-m-1$.
Therefore, its singular locus has dimension at most $d-m-2$.
\end{proof}

\begin{lemma}\label{lemma: pcl gamma j is ideal-theoretic complete inters}
The polynomials $G^h_1,\ldots,G^h_m$ and $\Delta^{i-1}F(\boldsymbol
A_0,T_1,\ldots,T_i)^h$ $(1\leq i \leq r)$ generate $J^h$. Hence,
$\mathrm{pcl}(\Gamma_{r}^*)$ is an ideal--theoretic complete
intersection of dimension $d-m$ and multidegree
$(d_1,\ldots,d_m,d,\ldots,d-r+1)$.
\end{lemma}
\begin{proof}
According to \cite[Theorem 18.15]{Eisenbud95}, it suffices to prove
that of the set of points of $\mathrm{pcl}(\Gamma_r^*)$ for which
the Jacobian of $G^h_1,\ldots,G^h_m$ and $\Delta^{i-1}F(\boldsymbol
A_0,T_1,\ldots,T_i)^h$ $(1\leq i \leq r)$ does not have full rank,
has codimension at least one in $\mathrm{pcl}(\Gamma_r^*)$. The set
of points of $\{T_0\not=0\}$ for which such a Jacobian does not have
full rank, has codimension one, because $G_1,\ldots,G_m$ and
$\Delta^{i-1}F(\boldsymbol A_0,T_1,\ldots,T_i)$ $(1\leq i \leq r)$
define a radical ideal. On the other hand, the set of points of
$\mathrm{pcl}(\Gamma_r^*)\cap\{T_0=0\}$ has codimension one in
$\mathrm{pcl}(\Gamma_r^*)$. This proves the first assertion of the
lemma.

We deduce that $\mathrm{pcl}(\Gamma_{r}^*)$ is an ideal--theoretic
complete intersection of dimension $d-m$, and the B\'ezout theorem
proves that $\deg\mathrm{pcl}(\Gamma_{r}^*)=\prod_{i=1}^md_i\cdot
d!/{(d-r)!}$.
\end{proof}

Finally, we prove the main result of this section.
\begin{theorem}\label{th: pcl Gamma_j is normal abs irred}
For $q>d\ge m+2$, $\mathrm{pcl}(\Gamma_{r}^*)\subset \Pp^{d+r}$ is a
normal ideal--theoretic complete intersection of dimension $d-m$ and
multidegree $(d_1,\ldots,d_m,d,\ldots,d-r+1)$.
\end{theorem}
\begin{proof}
Lemma \ref{lemma: pcl gamma j is ideal-theoretic complete inters}
shows that $\mathrm{pcl}(\Gamma_r^*)$ is an ideal--theoretic
complete intersection of dimension $d-m$ and multidegree
$(d_1,\ldots,d_m,d,\ldots,d-r+1)$. On the other hand, Theorem
\ref{th: codimension singular locus Gamma j} and Lemma \ref{lemma:
dimension of singular locus at infinity} show that the singular
locus of $\mathrm{pcl}(\Gamma_{r}^*)$ has codimension at least $2$
in $\mathrm{pcl}(\Gamma_{r}^*)$. This implies
$\mathrm{pcl}(\Gamma_{r}^*)$ is regular in codimension $1$ and thus
normal.
\end{proof}

By Theorems \ref{theorem: normal complete int implies irred} and
\ref{th: pcl Gamma_j is normal abs irred} we conclude that
$\mathrm{pcl}(\Gamma_{r}^*)$ is absolutely irreducible of dimension
$d-m$, and the same holds for $\Gamma_r^* \subset \A^{d+r}$. Since
$\Gamma_r$ is a nonempty Zariski open subset of $\Gamma_r^*$ of
dimension $d-m$ and $\Gamma_r^*$ is absolutely irreducible, the
Zariski closure of $\Gamma_r$ is $\Gamma_r^*$.
%

\section{The number of $\fq$--rational points of $\Gamma_{r}$}
\label{sec: number of points Gamma_j}
As before, let $d$, $m$ be positive integers with $q>d\ge m+2$. Let
$A_{d-1},\ldots,A_1$ be indeterminates over $\fq$, set $\boldsymbol
A:=(A_{d-1},\ldots,A_1)$, and let $G_1, \ldots, G_m$ be polynomials
of $\fq[A_{d-1},\ldots,A_1]$ satisfying hypotheses $({\sf H_1})$,
$({\sf H_2})$, $({\sf H_3})$ and $({\sf H_4})$. Let $\boldsymbol
G:=(G_1,\ldots,G_m)$ and $\mathcal{A}:=\mathcal{A}(\boldsymbol G)$
be the family defined as
$$\mathcal{A}:=\{T^d+a_{d-1}T^{d-1}+\cdots+a_1 T \in \fq[T]:
\,\boldsymbol G(a_{d-1},\ldots,a_1)=\boldsymbol 0\}.$$
%
%
In this section we determine the asymptotic behavior of the average
value set $\mathcal{V}(\mathcal{A})$ of $\mathcal{A}$.
By Lemma \ref{lemma: reduction to interp sets} we have
$$\mathcal{V}(\mathcal{A})= \frac{1}{|\mathcal{A}|}\sum_{r=1}^{d}
   (-1)^{r-1}\sum_{\mathcal{X}_r\subset \fq}\vert S_{\mathcal{X}_r}^{\mathcal{A}}\vert,
$$
where the second sum runs through all the subsets
$\mathcal{X}_r\subset\fq$ of cardinality $r$ and
$S_{\mathcal{X}_r}^{\mathcal{A}}$ denotes the number of
$f+a_0\in\mathcal{A}+\fq$ such that $(f+a_0)(\alpha)=0$ for any
$\alpha\in\mathcal{X}_r$.

Let $S_{r}^{\mathcal{A}}:=\sum_{\mathcal{X}_r\subset \fq}\vert
S_{\mathcal{X}_r}^{\mathcal{A}}\vert$. According to Lemmas
\ref{lemma: relacion entre gamma m,n y S m,n} and \ref{lemma:
relacion gamma_j y gamma_j estrella}, for $1\leq r \leq d$ we have
$$S_{r}^{\mathcal{A}}=\frac{|\Gamma_{r}(\fq)|}{r!}=
\frac{1}{r!}\bigg|\Gamma_r^*(\fq)\setminus\bigcup_{i\not=j}\{T_i=T_j\}\bigg|.$$
We shall apply the results on the geometry of $\Gamma_r^*$ of the
previous section in order to estimate the number of $\fq$--rational
points of $\Gamma_r^*$.
%
%
\subsection{An estimate for $S_{r}^{\mathcal{A}}$}
We shall rely on the following estimate (\cite[Theorem
1.3]{CaMaPr15}; see also \cite{GhLa02a}, \cite{CaMa07} or
\cite{MaPePr16} for similar estimates): if $W\subset \Pp^n$ is a
normal complete intersection defined over $\fq$ of dimension $l\geq
2$ and multidegree $(e_1,\ldots,e_{n-l})$, then
\begin{equation}\label{eq: estimate normal var CaMaPr}
\big||W(\fq)|-p_l\big| \leq
(\delta_W(D_W-2)+2)q^{l-\frac{1}{2}}+14D_W^2\delta_W^2 q^{l-1},
\end{equation}
where $p_l:=q^l+q^{l-1}+\cdots+1$, $\delta_W:=\prod_{i=1}^{n-l}e_i$
and $D_W:=\sum_{i=1}^{n-l}(e_i-1)$.

In what follows, we shall use the following notations:
$$\begin{array}{lll}
\displaystyle\delta_V:=\prod_{i=1}^md_i,&\,\,\displaystyle
\delta_{\Delta_r}:=\prod_{i=1}^r(d-i+1)=\frac{d!}{(d-r)!},&\displaystyle
\,\,\delta_r:=\delta_V\delta_{\Delta_r},\\
\displaystyle D_V:=\sum_{i=1}^m(d_i-1),& \displaystyle
D_{\Delta_r}:=\sum_{i=1}^r(d-i)=rd-\frac{r(r+1)}{2},&
D_r:=D_V+D_{\Delta_r}.
\end{array}
$$
We start with an estimate on the number of elements of
$\mathcal{A}$.
\begin{lemma}\label{lemma: lower bound |A|}
For $q>16(D_V\delta_V+14D_V^2\delta_V^2 q^{-\frac{1}{2}})^2$, we
have
$$\frac{1}{2}q^{d-m-1}<|\mathcal{A}|\le q^{d-m-1}+
2\big(\delta_V(D_V-2)+2+14D_V^2\delta_V^2q^{-\frac{1}{2}}\big)q^{d-m-\frac{3}{2}}.$$
\end{lemma}
\begin{proof}
Hypothesis $(\sf H_1)$, $(\sf H_2)$ and $(\sf H_3)$ imply that both
the projective closure $\mathrm{pcl}(V)\subset\Pp^{d-1}$ of $V$ and
the set $\mathrm{pcl}(V)^\infty:=\mathrm{pcl}(V)\cap\{T_0=0\}$ of
points at infinity are normal complete intersections defined over
$\fq$, both of multidegree $(d_1,\ldots, d_m)$ in $\Pp^{d-1}$ and
$\{T_0=0\}\cong\Pp^{d-2}$ respectively. Therefore, by \eqref{eq:
estimate normal var CaMaPr} it follows that
\begin{align*}
\big||\mathcal{A}|-q^{d-m-1}\big|=&
\big||\mathrm{pcl}(V)(\fq)|-|\mathrm{pcl}(V)^{\infty}(\fq)|-p_{d-m-1}+p_{d-m-2}\big|\nonumber\\
\leq & \big||\mathrm{pcl}(V)(\fq)|-p_{d-m-1}\big|+\big||\mathrm{pcl}(V(\fq))^{\infty}|-p_{d-m-2}\big| \nonumber\\
\le &
\left(\delta_V(D_V-2)+2\right)(q+1)q^{d-m-\frac{5}{2}}+14D_V^2\delta_V^2(q+1)q^{d-m-3}
\nonumber\\
\le &
2\big(\delta_V(D_V-2)+2+14D_V^2\delta_V^2q^{-\frac{1}{2}}\big)q^{d-m-\frac{3}{2}}.
\end{align*}
By the hypothesis on $q$ of the statement, the lemma readily
follows.
\end{proof}

By Theorem \ref{th: pcl Gamma_j is normal abs irred},
$\mathrm{pcl}(\Gamma_{r}^*)\subset\Pp^{d+r}$ is a normal complete
intersection defined over $\fq$ of dimension $d-m$ and multidegree
$(d_1,\ldots,d_m,d,\ldots,d-r+1)$. Therefore, applying (\ref{eq:
estimate normal var CaMaPr}) we obtain
$$
\big||\mathrm{pcl}(\Gamma_{r}^*)(\fq)|-p_{d-m}\big| \leq
(\delta_r(D_r-2)+2)q^{d-m-\frac{1}{2}}+14D_r^2\delta_r^2q^{d-m-1}.
$$
On the other hand, since
$\mathrm{pcl}(\Gamma_{r}^*)^{\infty}:=\mathrm{pcl}(\Gamma_{r}^*)\cap
\{T_0=0\}\subset\Pp^{d+r-1}$ is a finite union of at most $r+1$
varieties, each of pure dimension $d-m-1$ and degree $\delta_V$, by
\eqref{eq: upper bound -- projective gral} we have
$|\mathrm{pcl}(\Gamma_{r}^*)^{\infty}(\fq)|\le (r+1)\delta_V
p_{d-m-1}$. Hence,
\begin{align}\label{eq: estimate q-points Gamma_j estrella}
\big||\Gamma_{r}^*(\fq)|-&q^{d-m}\big|\leq
\big||\mathrm{pcl}(\Gamma_{r}^*)(\fq)|-p_{d-m}\big|+
\big||\mathrm{pcl}(\Gamma_{r}^*(\fq))^{\infty}|-p_{d-m-1}\big| \nonumber\\
\le &
\left(\delta_r(D_r-2)+2\right)q^{d-m-\frac{1}{2}}+14D_r^2\delta_r^2q^{d-m-1}+(r+1)\delta_V
p_{d-m-1}
\nonumber\\
\le &
\left(\delta_r(D_r-2)+2\right)q^{d-m-\frac{1}{2}}+(14D_r^2\delta_r^2+4r\delta_V)q^{d-m-1}.
\end{align}

We also need an upper bound on the number $\fq$--rational points of
$$
\Gamma_{r}^{*,=}:= \Gamma_{r}^* \bigcap \bigcup_{1 \leq i <j \leq r}
\{T_i =T_j\}.
$$
We observe that $\Gamma_{r}^{*,=}=\Gamma_{r}^{*}\cap \mathcal{H}_r$,
where $\mathcal{H}_r\subset \A^{d+r}$ is the hypersurface defined by
the polynomial $F_r:=\prod_{1 \leq i <j\leq r} (T_i-T_j)$. From the
B\'ezout inequality (\ref{eq: Bezout}) it follows that
\begin{equation}\label{eq: deg Gamma =}
\deg \Gamma_{r}^{*,=} \leq \delta_r \binom {r}{2}.
\end{equation}
Furthermore, we claim that $\Gamma_{r}^{*,=}$ has dimension at most
$d-m-1$. Indeed, let $(\boldsymbol{a}_0,\boldsymbol{\alpha})$ be any
point of $\Gamma_r^{*,=}$. Assume without loss of generality that
$\alpha_1=\alpha_2$. By the definition of divided differences we
deduce that $f_{\boldsymbol a_0}'(\alpha_1)=0$, which implies that
$f_{\boldsymbol a_0}$ has multiple roots. By the proof of Corollary
\ref{coro: J is radical}, the set of points $(\boldsymbol
a_0,\boldsymbol\alpha)$ of $\Gamma_r^*$ for which $f_{\boldsymbol
a_0}$ has multiple roots is contained in a subvariety of
$\Gamma_r^*$ of codimension 1, which proves the claim.

Combining the claim with (\ref{eq: deg Gamma =}) and \eqref{eq:
upper bound -- affine gral}, we obtain
\begin{equation}\label{eq: upper bound q-points gamma_j=}
\big|\Gamma_{r}^{*,=}(\fq)\big| \leq \delta_r\binom{r}{2} q^{d-m-1}.
\end{equation}
Since $\Gamma_{r}(\fq)=\Gamma_{r}^*(\fq)\setminus
\Gamma_{r}^{*,=}(\fq)$, from (\ref{eq: estimate q-points Gamma_j
estrella}) and (\ref{eq: upper bound q-points gamma_j=}) we deduce
that
\begin{align*}
\big||\Gamma_{r}(\fq)|&- q^{d-m}\big|\leq \big||\Gamma_{r}^*(\fq)|-
q^{d-m}\big|+
|\Gamma_{r}^{*,=}(\fq)|\\
\leq&\left(\delta_r(D_r-2)+2\right)q^{d-m-\frac{1}{2}}+\left(14D_r^2\delta_r^2+\binom{r}{2}\delta_r+4r
\delta_V\right)q^{d-m-1}.
\end{align*}
As a consequence, we obtain the following result.
\begin{theorem}\label{th: estimate S_j}
Let $q>d\ge m+2$. For any $r$ with and $1\leq r\leq d$, we have
$$
\bigg|S_{r}^{\mathcal{A}}- \frac{q^{d-m}}{r!}\bigg|\leq
\Bigg(\frac{\delta_r(D_r-2)+2}{r!}\,q^{\frac{1}{2}}+
\left(14\frac{D_r^2\delta_r^2}{r!}+
\binom{r}{2}\frac{\delta_r}{r!}+\frac{4r}{r!}\delta_V\right)\Bigg)q^{d-m-1}.$$
\end{theorem}
%
%
\subsection{An estimate for the average value set $\mathcal{V}(\mathcal{A})$}
Theorem \ref{th: estimate S_j} is the critical step in our approach
to estimate $\mathcal{V}(\mathcal{A})$.
\begin{corollary}\label{coro: average value sets}
With assumptions as in Lemma \ref{lemma: lower bound |A|} and
Theorem \ref{th: estimate S_j},
\begin{equation}\label{eq: estimate V(A)}
\left|\mathcal{V}(\mathcal{A})-\mu_d q\right|\le
2^d\delta_V(3D_V+d^2) q^{1/2}+
\frac{7}{4}\delta_V^2D_V^2d^4\sum_{k=0}^{d-1}\binom{d}{k}^{2}(d-k)!.
\end{equation}
\end{corollary}
\begin{proof}
According to Lemma \ref{lemma: reduction to interp sets}, we have
\begin{align}
\mathcal{V}(\mathcal{A})&-\mu_d\, q=
\frac{1}{|\mathcal{A}|}\sum_{r=1}^{d}
(-1)^{r-1}\left(S_{r}^{\mathcal{A}}
-\frac{|\mathcal{A}|q}{r!}\right)\nonumber\\
&= \frac{1}{|\mathcal{A}|}\sum_{r=1}^{d}
(-1)^{r-1}\left(S_{r}^{\mathcal{A}}-\frac{q^{d-m}}{r!}\right)
-\frac{1}{|\mathcal{A}|}\sum_{r=1}^{d}
(-1)^{r-1}\left(\frac{|\mathcal{A}|q-q^{d-m}}{r!}\right)\nonumber\\
&= \frac{1}{|\mathcal{A}|}\sum_{r=1}^{d}
(-1)^{r-1}\left(S_{r}^{\mathcal{A}}-\frac{q^{d-m}}{r!}\right)
+\mu_d\left(\frac{q^{d-m}-|\mathcal{A}|q}{|\mathcal{A}|}\right)\label{eq:
average value set th 2}.
\end{align}

We consider the absolute value of the first sum in the right--hand
side of (\ref{eq: average value set th 2}). From Lemma \ref{lemma:
lower bound |A|} and Theorem \ref{th: estimate S_j} we have
\begin{align*}
\frac{1}{|\mathcal{A}|}\sum_{r=1}^{d}
\bigg|&S_{r}^{\mathcal{A}}-\frac{q^{d-m}}{r!}\bigg| \le
\frac{2}{q^{d-m-1}}\sum_{r=1}^{d}
\left|S_{r}^{\mathcal{A}}-\frac{q^{d-m}}{r!}\right|\\
&\le 2q^{\frac{1}{2}}\sum_{r=1}^{d}\frac{\delta_r(D_r-2)+2}{r!}\,+
28\sum_{r=1}^{d}\frac{D_r^{2}\delta_r^{2}}{r!}+
2\delta_V\sum_{r=1}^{d}\frac{\binom{r}{2}\delta_{\Delta_r}+4r}{r!}.
\end{align*}
Concerning the first sum in the right--hand side, we see that
\begin{align*}
\sum_{r=1}^{d}\frac{\delta_r(D_r-2)+2}{r!}&\le\delta_V\Bigg(D_V
\sum_{r=1}^{d}\binom{d}{r}+\sum_{r=1}^{d}\frac{\delta_{\Delta_r}(D_{\Delta_r}-2)+2}{r!}\Bigg)\\
&\leq \delta_V\big(D_V2^d+d^22^{d-1}\big)=
2^{d-1}\delta_V(2D_V+d^2).
\end{align*}
On the other hand,
\begin{align*}
\sum_{r=1}^{d}\frac{D_r^2\delta_r^{2}}{r!}\!&=
\delta_V^2\Bigg(D_V^2\sum_{r=1}^{d}\frac{\delta_{\Delta_r}^2}{r!}
+2D_V\sum_{r=1}^{d}\frac{D_{\Delta_r}\delta_{\Delta_r}^2}{r!}+\sum_{r=1}^{d}
\frac{D_{\Delta_r}^2\delta_{\Delta_r}^2}{r!}\Bigg)\\
&\le \delta_V^2\bigg(\frac{D_V^2}{4} +D_V+1\bigg)\sum_{r=1}^{d}
\frac{D_{\Delta_r}^2\delta_{\Delta_r}^2}{r!}\\ &
\leq\delta_V^2\frac{(D_V^2+2)^2}{4}
\frac{1}{64}(2d-1)^{4}\sum_{k=0}^{d-1}\binom{d}{k}^{2}(d-k)!.
\end{align*}
Finally, we consider the last sum
$$
\sum_{r=1}^{d}\frac{\delta_{\Delta_r}}{2(r-2)!}=
\sum_{r=1}^d\binom{d}{r}\frac{r(r-1)}{2}=\sum_{k=0}^{d-1}\binom{d}{k}\frac{(d-k)!}{2\,(d-k-2)!}.
$$
As a consequence, we obtain
\begin{align*}
\frac{1}{|\mathcal{A}|}\sum_{r=1}^{d}
\left|S_{r}^{\mathcal{A}}-\frac{q^{d-m}}{r!}\right|\!\le
&q^{\frac{1}{2}}2^d\delta_V(2D_V+d^2)+
\frac{7}{64}(2d-1)^4\sum_{k=0}^{d-1}\binom{d}{k}^{2}(d-k)!\\&+
\sum_{k=0}^{d-1}\binom{d}{k}(d-k)!+8\sum_{k=0}^{d-1}\frac{1}{(d-k-1)!}.
\end{align*}

Concerning the second sum in the right--hand side of \eqref{eq:
average value set th 2}, by Lemma \ref{lemma: lower bound |A|} it
follows that
$$\left|\frac{q^{d-m}-|\mathcal{A}|q}{|\mathcal{A}|}\right|\le
4\big(\delta_V(D_V-2)+2+14D_V^2\delta_V^2q^{-\frac{1}{2}}\big)q^{\frac{1}{2}}.$$
The statement of the corollary follows by elementary calculations.
\end{proof}

Next we analyze the behavior of the right--hand side of (\ref{eq:
estimate V(A)}). This analysis consists of elementary calculations,
which are only sketched.

Fix $k$ with $0\le k\le d-1$ and denote
$h(k):=\binom{d}{k}^2(d-k)!$. From an analysis of the sign of the
differences $h(k+1)-h(k)$ for $0\le k\le d-1$ we deduce the
following remark, which is stated without proof.
\begin{remark}\label{rem: growth h(k)}
Let $k_0:=-1/2+\sqrt{5+4d}/2$. Then $h$ is either an increasing
function or a unimodal function in the integer interval $[0,d-1]$,
which reaches its maximum at $\lfloor k_0\rfloor$.
\end{remark}

From Remark \ref{rem: growth h(k)} we see that
\begin{equation}\label{eq: expresion a analizar}
\sum_{k=0}^{d-1}\binom{d}{k}^{2}(d-k)!\le d \binom{d}{\lfloor
k_0\rfloor}^2(d-\lfloor k_0\rfloor)!=\frac{d\,(d!)^2}{(d-\lfloor
k_0\rfloor)!\,(\lfloor k_0\rfloor!)^2 }.
\end{equation}
Now we use the following version of the Stirling formula (see, e.g.,
\cite[p. 747]{FlSe09}): for $m\in \N$, there exists $\theta$ with
$0\le \theta<1$ such that
$$m!=(m/e)^m\sqrt{2\pi m}\,e^{\theta/12m}.$$
By the Stirling formula there exist $\theta_i$ ($i=1,2,3$) with
$0\le \theta_i<1$ such that
$$C(d)\!:=\frac{d\,(d!)^2}{(d-\lfloor
k_0\rfloor)!\,(\lfloor k_0\rfloor!)^2 }\le \frac{d\,
d^{2d+1}e^{-d+\lfloor
k_0\rfloor}e^{\frac{\theta_1}{6d}-\frac{\theta_2}{12(d-\lfloor
k_0\rfloor)}-\frac{\theta_3}{6\lfloor k_0\rfloor}}}{\big(d-\lfloor
k_0\rfloor\big)^{d-\lfloor k_0\rfloor}\sqrt{2\pi(d-\lfloor
k_0\rfloor)}\lfloor k_0\rfloor^{2\lfloor k_0\rfloor+1}}.$$
By elementary calculations we obtain
\begin{align*}
 (d-\lfloor k_0\rfloor)^{-d+\lfloor k_0\rfloor}&\le
 d^{-d+\lfloor k_0\rfloor}e^{\frac{\lfloor k_0\rfloor(d-\lfloor
 k_0\rfloor)}{d}},\quad
 \frac{ d^{\lfloor k_0\rfloor}}{{\lfloor k_0\rfloor}^{2\lfloor k_0\rfloor}} \le
 e^{\frac{d-\lfloor k_0\rfloor^2}{\lfloor k_0\rfloor}}.
\end{align*}
It follows that
$$C(d)\le \frac{d^{d+2}e^{2 \lfloor
k_0\rfloor}e^{-\frac{\lfloor k_0\rfloor^2}{d} +
\frac{1}{6d}+\frac{d-\lfloor k_0 \rfloor ^2}{\lfloor
k_0\rfloor}}}{\sqrt{2\pi} e^d \sqrt{d-\lfloor k_0\rfloor}\lfloor
k_0\rfloor}.
$$

By the definition of $\lfloor k_0\rfloor$, it is easy to see that
${d}/{\lfloor k_0\rfloor\sqrt{d-\lfloor k_0\rfloor}}\le {5}/{2}$ and
that $2\lfloor{k_0}\rfloor \le -1+\sqrt{5+4d}\le -1/5+2\sqrt{d}$.
Therefore, taking into account that $d\ge 2$, we conclude that
$$
C(d)\le \frac{5}{2}\,\frac{e^{\frac{109}{30}} d^{d+1}e^{2\sqrt{d}}}{
\sqrt{2\pi }e^d}.
$$
Combining this bound with Corollary \ref{coro: average value sets}
we obtain the following result.
\begin{theorem}\label{theorem: final main result - mean}
For $q>\max\big\{d,16(D_V\delta_V+14D_V^2\delta_V^2
q^{-\frac{1}{2}})^2\}$ and $d\ge m+2$, the following estimate holds:
$$
\left|\mathcal{V}(\mathcal{A})-\mu_d\,q\right|\le
2^d\delta_V(3D_V+d^2) q^{\frac{1}{2}} +
67 \delta_V^2(D_V+2)^2\,{d^{d+5} e^{2 \sqrt{d}-d}}. $$
\end{theorem}
%
%
\subsection{Applications of our main result}
We discuss two families of examples where hypotheses ${\sf (H_1)}$,
${\sf (H_2)}$, ${\sf (H_3)}$ and ${\sf (H_4)}$ hold. Therefore, the
estimate of Theorem \ref{theorem: final main result - mean} is valid
for these families.

Our first example concerns linear families of polynomials. Let
$L_i\in\fq[A_{d-1},\ldots,A_2]$ be polynomials of degree
1 $(1\le i\le m)$. Assume without loss of generality that the Jacobian matrix of
$L_1,\ldots,L_m$ with respect to $A_{d-1},\ldots,A_2$ is of full
rank $m\le d-2$. Consider the family $\mathcal{A}_\mathcal{L}$
defined as
$$
\mathcal{A}_\mathcal{L}\!:=\!\big\{T^d+a_{d-1}T^{d-1}+\cdots+a_0\!\in
\fq[T]: L_i(a_{d-1},\ldots,a_2)=0\, (1\le i\le m)\big\}.
$$

It is clear that hypotheses ${\sf (H_1)}$, ${\sf (H_2)}$ and ${\sf
(H_3)}$ hold. Now we analyze the validity of ${\sf (H_4)}$. Denote
by $\mathcal{L}\subset\A^d$ the linear variety defined by
$L_1,\ldots,L_m$, and let $\mathcal{D}(\mathcal{L})\subset\A^d$ and
$\mathcal{S}_1(\mathcal{L})\subset\A^d$ be the discriminant locus
and the first subdiscriminant locus of $\mathcal{L}$. Since the
coordinate ring $\cfq[\mathcal{L}]$ is a domain, hypothesis ${\sf
(H_4)}$ holds if the coordinate function defined by the discriminant
$\mathrm{Disc}(F(\boldsymbol A_0,T))$ in $\cfq[\mathcal{L}]$ is
nonzero, and the class of the subdiscriminant
$\mathrm{Subdisc}(F(\boldsymbol A_0,T))$ in the quotient ring
$\cfq[\mathcal{L}]/\mathrm{Disc}(F(\boldsymbol A_0,T))$ is not a
zero divisor. For fields $\fq$ of characteristic $p$ not dividing
$d(d-1)$, both assertions are consequences of \cite[Theorem
A.3]{MaPePr14}. Taking into account that $\delta_\mathcal{L}=1$ and
$D_\mathcal{L}=0$, applying Theorem \ref{theorem: final main result
- mean} we obtain the following result.
\begin{theorem}\label{theorem: application linear families}
For $p:=\mathrm{char}(\fq)$ not dividing $d(d-1)$ and $q>d\ge m+2$,
$$
\left|\mathcal{V}(\mathcal{A}_\mathcal{L})-\mu_d\,q\right|\le 2^dd^2
q^{\frac{1}{2}} + 268{d^{d+5} e^{2 \sqrt{d}-d}}. $$
\end{theorem}

Our second example consists of a nonlinear family of polynomials.
Let $s$, $m$ be positive integers with $m\le s\le d-m-4$, let
$\Pi_1,\ldots,\Pi_s$ be the first $s$ elementary symmetric
polynomials of $\fq[A_{d-1},\ldots,A_2]$ and let
$G_1,\ldots,G_m\in\fq[A_{d-1},\ldots,A_2]$ be symmetric polynomials
of the form $G_i:=S_i(\Pi_1,\ldots,\Pi_s)$ $(1\le i\le m)$. Consider
the weight function ${\sf wt}:\fq[Y_1,\ldots,Y_s]\to\N$ defined by
setting ${\sf wt}(Y_i):=i$ $(1\le i\le s)$ and denote by $S_1^{\sf
wt},\ldots,S_m^{\sf wt}$ the components of highest weight of
$S_1,\ldots,S_m$. Assume that both $S_1,\ldots,S_m$ and $S_1^{\sf
wt},\ldots,S_m^{\sf wt}$ form regular sequences of
$\fq[Y_1,\ldots,Y_s]$, and the Jacobian matrices of $S_1,\ldots,S_m$
and $S_1^{\sf wt},\ldots,S_m^{\sf wt}$ with respect to
$Y_1,\ldots,Y_s$ have full rank in $\A^s$. We remark that varieties
defined by polynomials of this type arise in several combinatorial
problems over finite fields (see, e.g., \cite{CaMaPr12},
\cite{CeMaPePr14}, \cite{MaPePr14}, \cite{CeMaPe15} and
\cite{MaPePr15}). Finally, let
%
$$
\mathcal{A}_\mathcal{N}\!:=\!\big\{T^d+a_{d-1}T^{d-1}+\cdots+a_0\!\in
\fq[T]: G_i(a_{d-1},\ldots,a_2)=0\, (1\le i\le m)\big\}.
$$

Hypotheses ${\sf (H_1)}$, ${\sf (H_2)}$ and ${\sf (H_3)}$ hold due
to general facts of varieties defined by symmetric polynomials (see
\cite{MaPePr15} for details). Further, it can be shown that ${\sf
(H_4)}$ holds by a generalization of the arguments proving the
validity of ${\sf (H_4)}$ for the linear family
$\mathcal{A}_\mathcal{L}$. As a consequence, applying Theorem
\ref{theorem: final main result - mean} we deduce the following
result.
\begin{theorem}\label{theorem: application nonlinear families}
For $p:=\mathrm{char}(\fq)$ not dividing $d(d-1)$, $m\le s\le d-m-4$
and $q>\max\big\{d,16(D_\mathcal{N}\delta_\mathcal{N}
+14D_\mathcal{N}^2\delta_\mathcal{N}^2 q^{-\frac{1}{2}})^2\big\}$,
where $\delta_\mathcal{N}:=\prod_{i=1}^md_i$ and
$D_\mathcal{N}:=\sum_{i=1}^m(d_i-1)$, the following estimate holds:
$$
\left|\mathcal{V}(\mathcal{A}_\mathcal{N})-\mu_d\,q\right|\le
2^d\delta_\mathcal{N}(3D_\mathcal{N}+d^2) q^{\frac{1}{2}} + 67
\delta_\mathcal{N}^2(D_\mathcal{N}+2)^2\,{d^{d+5} e^{2 \sqrt{d}-d}}.
$$
\end{theorem}

%
%
%
\bibliographystyle{amsalphaInitialsNoOx}
\bibliography{refs1,finite_fields}

\end{document}